\documentclass[journal]{IEEEtran}
\usepackage{latexsym,url,amscd,amsmath,amssymb}
\usepackage{graphics}
\usepackage{graphicx}
\usepackage{epstopdf}
\usepackage[usenames]{color}
\usepackage{multirow}
\usepackage{multicol}
\usepackage{pgf,pgfarrows,pgfnodes,pgfautomata,pgfheaps}
\usepackage[normalem]{ulem}
\usepackage{verbatim}
\usepackage{subfigure}
\usepackage{tabularx}
\usepackage{dsfont}
\usepackage{framed}

\DeclareMathOperator*{\argmin}{arg\,min}

\DeclareMathOperator*{\dom}{dom}

\def\A{{\mathcal A}}
\def\B{{\mathcal B}}
\def\C{{\mathcal C}}
\def\D{{\mathcal D}}
\def\E{{\mathcal E}}

\def\H{{\mathcal H}}
\def\I{{\mathcal I}}
\def\K{{\mathcal K}}
\def\L{{\mathcal L}}
\def\M{{\mathcal M}}
\def\N{{\mathcal N}}

\def\Q{{\mathcal Q}}

\def\S{{\mathcal S}}
\def\T{{\mathcal T}}

\def\X{{\mathcal X}}
\def\Y{{\mathcal Y}}
\def\Z{{\mathcal Z}}

\newtheorem{lemma}{Lemma}[section]

\newtheorem{theorem}{Theorem}[section]
\newtheorem{assumption}{Assumption}[section]

\newtheorem{proposition}{Proposition}[section]

\renewcommand{\thesection}{\arabic{section}.}
\renewcommand{\thesubsection}{\arabic{section}.\arabic{subsection}.}
\renewcommand{\theequation}{\arabic{section}.\arabic{equation}}

\title {Symmetric Gauss-Seidel Technique Based Alternating Direction Methods of Multipliers for Transform Invariant
Low-Rank Textures Problem}
\author{Yanyun Ding\thanks{Y. Ding is with School of Mathematics and Statistics,
Henan University, Kaifeng 475000, China (Tel: +86-37123881696, Email: dingyanyunhenu@163.com).}
\ and \
Yunhai Xiao\thanks{Y. Xiao is with Institute of Applied Mathematics, School of Mathematics and Statistics,
Henan University, Kaifeng 475000, China (Tel: +86-37123881696, Email: yhxiao@henu.edu.cn, yhxiaomath@gmail.com). The author's work is supported by the Major State Basic Research
Development Program of China (973 Program) (Grant No. 2015CB856003),
and the National Natural Science Foundation of China (Grant No. 11471101).} }

\begin{document}
\maketitle

\begin{abstract}
Transform Invariant Low-Rank Textures, referred to as TILT, can accurately and robustly extract textural or geometric information in a 3D  from user-specified windows in 2D in spite of significant corruptions and warping.
It was discovered that the task can be characterized, both theoretically and
numerically, by solving a sequence of matrix nuclear-norm and $\ell_1$-norm involved
convex minimization problems. For solving this problem, the direct extension of Alternating
Direction Method of Multipliers (ADMM) in an usual Gauss-Seidel manner often performs numerically well in practice but there is no theoretical guarantee on its convergence. In this paper, 
we resolve this dilemma by using the novel symmetric Gauss-Seidel (sGS) based ADMM developed by Li, Sun \& Toh (Math. Prog. 2016). The sGS-ADMM is guaranteed to converge and we shall demonstrate in this paper that it is also practically efficient than the directly extended ADMM.
When the sGS technique is applied to this particular problem, we show that only one variable needs to be re-updated, and this updating hardly imposes any excessive computational cost. The sGS decomposition theorem  of Li, Sun \& Toh (arXiv: 1703.06629) establishes the equivalent between sGS-ADMM and the classical ADMM with an additional semi-proximal term, so the convergence result is followed directly. Extensive experiments illustrate that
the sGS-ADMM and its generalized variant have superior numerical efficiency over the directly extended ADMM.
\end{abstract}


\setcounter{equation}{0}
\section{Introduction}\label{section1}
\IEEEPARstart{D}{etecting}, identifying, and recognizing feature points or salient regions in images is a very important and fundamental problem in computer vision. These points and regions carry rich and high-level semantic information which are important for image understanding.  Hence, extracting both textural and geometric information accurately may
facilitate many real-world applications such as camera calibration, 3D reconstruction, character recognition, and scene understanding.

Because different points or regions are often used to establish or measure the similarity between different images, it is hoped that the transformation occurred under the changes of viewpoint or illumination has some stability or invariance properties. For these reasons, many so-called invariant features and descriptors for capturing geometrically meaningful structures from various images have been proposed, analyzed, and implemented over the past decades.

Among these methods, the widely used type is the ``scale invariant feature transform" (SIFT) \cite{SIFT,ASIFT}, which is often invariant to the changes in rotation and scale within a limited extent, but it is not truly invariant under projective transforms \cite{TILT}. Unlike conventional techniques, the ``Transform Invariant Low-Rank Textures"  (TILT) \cite{TILT} correctly extracts rich structural and geometric information about the image in 3D scene from its 2D images, and simultaneously produces the global correlations or transformations of those regions in 3D, which are truly invariant of image domain transformations.

We consider a true 2D low-rank texture $X\in\mathbb{R}^{m\times n}$ lies on a planar surface in 3D scene. It is called a low-rank texture if $r\ll\min\{m,n\}$, where
$r\triangleq\text{rank}(X)$. All regular, symmetric patters clearly belong to this class of textures. The image that we observed from a certain viewpoint is actually a transformed version of the original low-rank texture $X$, i.e., $D=X\circ \tau^{-1}$, where $D$ is an observed image (deformed and corrupted) and $\tau:\mathds{R}^2\rightarrow\mathds{R}^2$ is a certain group of transforms, e.g.,  affine transforms, perspective transforms, and general cylindrical transforms \cite{zltwo}. Generally, the transformed texture $D$ might no longer be low rank in such a situation. But beyond that, the textures images are often corrupted by noises and occlusions, or contain some pixels from the surrounding background. Therefore, the following model is more faithful to real this situation
$$
D\circ\tau=X+E,
$$
where $E$ corresponds to the noises or errors.  We assume that, in this paper,
$E$ is a sparse matrix, which means that only a small fraction of the image pixels are grossly corrupted.
Our goal is to recover the exact low-rank texture $X$ and the domain
transformation $\tau$ from the observed image $D$, which naturally leads to the following optimization problem
\begin{equation}\label{modelone}
  \min_{X,E,\tau} \big\{\text{rank}(X)+\lambda\| E\|_{0}, \quad
  \text{s.t.}\quad D\circ\tau=X+E\big\},
\end{equation}
where  $\| E\|_{0}$ denotes the number of non-zero entries in $E$, and $\lambda>0$ is a
weighting parameter that balance the rank of the texture versus the sparsity of the error.
Actually, problem (\ref{modelone}) is combinatorial and known to be NP-hard, and generally computationally intractable. Therefore,
convex relaxations are often used to make the minimization tractable.

The most popular choice is to replace the ``$\text{rank}(\cdot)$" term with the nuclear norm \cite{FPHD}, and replace
the $\ell_{0}$-norm term with the $\ell_{1}$-norm \cite{CANDES11}, which yields the following convex minimization problem to produce an approximate solution
\begin{equation}\label{modeltwo}
  \min_{X,E,\tau} \big\{\|X\|_{*}+\lambda\| E\|_{1},\quad
  \text{s.t.}\quad D\circ\tau=X+E\big\},
\end{equation}
where $\|\cdot\|_*$ is the so-called nuclear norm (also known
as Ky Fan norm) defined by the sum of all singular values, and $\|\cdot\|_1$ is defined as the sum of absolute values of all entries. This model
is also derived from the batch images alignment problem by Pent {\itshape et al.} \cite{PENG12} to seek an optimal set of images domain transformations,
where $X$ represents a batch of aligned images and $E$ models the differences among images.
We must emphasize that although the objective function in model (\ref{modeltwo}) is convex and separable, the nonlinear constraint may cause many difficulties to minimize.
As mentioned in \cite{PENG12,TILT} that, a common technique to overcome this difficulty is to linearize the nonlinear term $D\circ\tau$  at the current estimation $\tau^{(i)}$ as
$D\circ(\tau^{(i)}+\Delta\tau)\approx D\circ\tau^{(i)}+J\Delta\tau$, and then compute the increment $\Delta\tau$ via solving a sequence of three-block convex minimization problem with form
\begin{equation}\label{modelthree}
    \min_{X,E,\Delta\tau} \big\{\| X\|_{*}+\lambda\|E\|_{1},\quad
  \text{s.t.}\quad D\circ\tau^{(i)}+J\Delta\tau=X+E\big\},
\end{equation}
where $J$ is the Jacobian of the image with respect to the transform parameters $\tau^{(i)}$  defined as
\begin{equation}\label{computJ}
J=\frac{\partial}{\partial \zeta}\Big (\frac{\text{vec}(D\circ\zeta)}{\|\text{vec}(D\circ\zeta)\|_2}\Big)\Big |_{\zeta=\tau^{(i)}},
\end{equation}
where ``$\text{vec}(\cdot)$" is used to stack a matrix column-by-column sequentially as a vector. When the increment $\Delta\tau$ is attained, the transform is immediately updated as
$\tau^{(i+1)}=\tau^{(i)}+\Delta\tau$. It is important to assume that $D\circ\tau^{(i)}$ does not belong to the rang space of $J$. Otherwise, problem (\ref{modelthree}) only admits zero solutions. 
The model (\ref{modelthree})  has separable
structure in terms of both the objective function and the constraint, and thus, it falls into the
framework of the alternating direction method of multipliers (ADMM).
Zhang {\itshape et al.} \cite{TILT} implemented the directly extended ADMM  and illustrated its practical performance. Nevertheless,
the directly extended ADMM is divergent for multi-block convex minimization problems, so its convergence can not be theoretically guaranteed \cite{CHENMP16}. Because of this, Ren \& Lin \cite{LIN13} reformulated problem (\ref{modelthree}) as a two-block convex minimization and solved immediately by a linearized ADMM with an adaptive penalty parameter updating technique.
In this paper, we also focus the application of  ADMM on the three-variable involved convex minimization (\ref{modelthree}) since this method has been widely and successfully used in the field of image processing, such as\cite{XIAOJMIV,XIAOIPI,YANGSIAMSC}.
However, unlike the aforementioned approaches, we employ a symmetric Gauss-Seidel (sGS) based ADMM developed by Li, Sun \& Toh \cite{XDLIMP} to sweep one of the variables just one more time. Due to the simple closed-form solutions are admitted for subproblems, this technique imposes almost no excessive computational burdens. The advantage of using the sGS technique is that it decomposes a large problem into several smaller parts and then solves it correspondingly via its favorable structure. The technique has been widely and successfully used to solve many multi-block conic programming problems over the past few years, such as \cite{CHENL,XDLIMP,LISGS,YANGSIAM}.
We show that the sGS decomposition theorem in \cite{SGSTH} can be used to establish the equivalence between the sGS-ADMM and the semi-proximal ADMM with a specially designed semi-proximal term which, allows the desired convergence to be directly derived from the convergence result of Fazel et al. \cite{SEMP13}.

The remaining parts of this paper is organized as follows. Section \ref{section2} contains two subsections. Subsection \ref{sub1} reviews some basic definitions and facts in convex analysis. Subsection \ref{sub2} reviews some typical ADMMs and the convergence results for our subsequent developments.
In Section \ref{section3}, we apply the sGS-ADMM to solve (\ref{modelthree}) and list its convergence result immediately. In Section \ref{theory}, we present an sGS based generalized ADMM method. In Section \ref{numersec}, we provide computational experiments to show the algorithms' practical performance. And finally we conclude the paper in
Section \ref{last}.

\section{Preliminaries}\label{section2}
\setcounter{equation}{0}
In this section, we provide some basic concepts and give a quick review
of a couple of semi-proximal ADMM which will be used in the subsequent developments.
\subsection{Basic concepts}\label{sub1}
Let $\E$ be finite dimensional real Euclidean space endowed with an inner product $\langle \cdot,\cdot \rangle $ and its induced norm $\|\cdot\|$, respectively.
A subset $\C$ of $\E$ is said to be convex if $(1-\lambda)x+\lambda y\in\C$ whenever $x\in\C$, $y\in\C$, and $0\leq\lambda\leq 1$. The relative interior of $C$, which we denote by $ri(\C)$, is defined as the interior which results when $\C$ is regarded as a subset of its affine hull. For any $z\in\E$, the symbol $\Pi_{\C}(z)$ denotes the metric projection of $z$ onto $\C$, which is the optimal solution of the minimization problem $\min_y\{ \|y-z\| | y\in\C\}$. A subset $\K$ of $\E$ is called a cone if it is closed under positive scalar multiplication, i.e., $\lambda x\in\K$ when $x\in\K$ and $\lambda>0$ \cite{RR}. The normal cone of $\K$ at point $x\in\K$ is defined by
$\N_\K(x)=\{y\in\E|\langle y, z-x\rangle\leq 0, \ \forall z\in\K\}$. 

Let $f:\E\rightarrow(-\infty,+\infty]$ be a closed proper convex function. The effective domain of $f$, which we denote by $\text{dom}(f)$, is defined
as $\text{dom}(f)=\{x| f(x)<+\infty\}$. A vector $x^*$ is said
to be a subgradient of $f$ at point $x$ if $f(z)\geq f(x)+\langle x^*,z-x\rangle$ for all $z\in\E$. The set of all subgradients of $f$ at $x$ is called the subdifferential of $f$ at $x$
and is denoted by $\partial f(x)$. Obviously, $\partial f(x)$ is a closed convex set while it is not empty. The multivalued operator $\partial f:x\rightrightarrows \partial f(x)$ is shown to be maximal monotone \cite[Corollary 31.5.2]{RR}, i.e., for any $x,y\in\E$ such that $\partial f(x)$ and $\partial f(y)$ are not empty, it holds that $\langle x-y, u-v\rangle\geq 0$ for all $u\in\partial f(x)$ and $v\in\partial f(y)$.
The Moreau-Yosida regularization of $f$ at $x\in\E$ with positive scalar $\beta>0$ is defined by
\begin{equation}\label{MY}
\varphi_f^\beta(x):=\min_{y \in \E}\Big\{f(y)+\frac{1}{2\beta}\|y-x\|^2\Big\}.
\end{equation}
For any $x \in \E$, problem (\ref{MY}) has a unique optimal solution, which is well
known as the proximal point of $x$ associated with $f$, i.e.,
\begin{align}\label{pi}
P_{f}^{\beta}(x):=\argmin_{y \in \E}\Big\{f(y)+\frac{1}{2\beta}\|y-x\|^2\Big\}.
\end{align}
The following propositions server as important building blocks in the subsequent developments:

\begin{proposition}\label{prop1}
(\cite[Theorem 2.1]{JFCAI}) Given $X\in\mathbb{R}^{m\times n}$ of rank $r$,
let
$$
X=U\Sigma V^\top,\quad\text{and}\quad\Sigma=\text{diag}(\{\sigma_i\}_{1\leq i\leq
r}),
$$
be the singular value decomposition (SVD) of $X$. For each $\mu>0$, it is shown that the proximal point of $X$ defined as
\begin{equation}\label{nclear}
\mathcal{D}_{\mu}(X)=\argmin_Y \Big\{\|Y\|_*+\frac{1}{2\mu}\|Y-X\|_F^2\Big\}
\end{equation}
can be characterized as follows
$$
\mathcal{D}_{\mu}(X)=U\Sigma_{\mu}V^\top\quad and \quad
\Sigma_\mu=\text{diag}(\{\sigma_i-\mu\}_+),
$$
where $\{\cdot\}_+=\max\{0,\cdot\}$.
\end{proposition}

\begin{proposition}\label{prop2}
Let $X\in\mathbb{R}^{m\times n}$ be a given matrix. For each $\mu>0$, the proximal point of $X$ is defined as
$$
\mathcal{S}_\mu(X)=\argmin_Y \|Y\|_{1}+\frac{1}{2\mu}\|X-Y\|_F^2.
$$
It is shown that the $(i,j)$-entry of $\mathcal{S}_\mu(X)$  can be characterized as follows
$$
[\mathcal{S}_\mu(X)]_{i,j}=\mbox{sgn} (X_{i,j}) \cdot\Big\{|X_{i,j}|-\mu\Big\}_+,
$$
where ``$\text{sgn}$" is sign function.
\end{proposition}

\subsection{Classical and Generalized Semi-proximal ADMM}\label{sub2}
Let $\X$, $\Y$, and $\Z$ be finite dimensional real Euclidian spaces. Consider the convex optimization problem with the following two-block separable structure
\begin{equation}\label{tss}
\begin{array}{ll}
\min\limits_{y,z} & f(y)+g(z) \\
\text{s.t.} & \mathcal{A}^*y+\mathcal{B}^*z=c,
\end{array}
\end{equation}
where $f : \Y \rightarrow (-\infty,+\infty]$ and $g : \Z \rightarrow (-\infty,+\infty]$ are closed proper convex functions, $\A: \X \rightarrow \Y$ and $\B : \X \rightarrow \Z$ are given linear maps, and $c \in \mathcal{X}$ is given data. The dual of problem (\ref{tss}) is given by
\begin{equation}\label{dss}
\max_{x} \Big\{f^*(-\A x)+g^*(-\B x)+\langle c,x \rangle\Big\}.
\end{equation}
The Karush-Kuhn-Tucker (KKT) system of problem (\ref{tss}) is given by
$$
0\in \A x+\partial f(y), \quad 0\in \B x+\partial g(z),\quad\mbox{and} \quad \A^*y+\B^*z=c.
$$

The augmented Lagrangian function associated with (\ref{tss}) is given by
$$
\L_{\sigma}(y,z;x)=f(y)+g(z)+\langle x,\mathcal{A}^*y+\mathcal{B}^*z-c\rangle+\frac{\sigma}{2}
\|\mathcal{A}^*y+\mathcal{B}^*z-c\|^2,
$$
where $x\in\X$ is a multiplier, and $\sigma >0$ be a given penalty parameter. Staring from an initial point $(x^{0},y^{0},z^{0})\in\X\times(\dom\, f)\times(\dom\, g)$, the iterations of the semi-proximal ADMM of Fazel, Pong, Sun \& Tseng \cite{SEMP13} for solving (\ref{tss}) is summarized as
\begin{equation}\label{sadmm}
\left\{
\begin{array}{ll}
y^{k+1}&=\argmin_{y}\big\{\L_{\sigma}(y,z^{k};x^{k})+\frac{\sigma}{2}\|y-y^k\|_{\T_f}^2\big\},
\\[2mm]
z^{k+1}&=\argmin_{z}\big\{\L_{\sigma}(y^{k+1},z;x^{k})+\frac{\sigma}{2}\|z-z^k\|_{\T_g}^2\big\},
\\[2mm]
x^{k+1}&=x^{k}+\xi\sigma\big(\A^{*}y^{k+1}+\B^{*}z^{k+1}-c\big),
\end{array}
\right.
\end{equation}
where $\T_f$ and $\T_g$ are positive semi-definite and the step-length $\xi$ is chosen in the interval $(0,(1+\sqrt{5})/2)$. On the one hand, when $\T_f=0$ and $\T_g=0$, the semi-proximal ADMM (\ref{sadmm}) reduces to the classical ADMM introduced by Glowinski \& Marroco \cite{GLOWINSKI75} and Gabay \& Mericire \cite{GABAY76} in the mid-1970s. On the other hand, when $\T_f=\alpha \I$ and $\T_g=\beta\I$ with positive scalars $\alpha>0$ and $\beta>0$, the iterative scheme (\ref{sadmm}) comes down to proximal ADMM presented by Eckstein \cite{ECKSTEINOMS} in 1990s.  The following theorem  is selected from the convergence Theorem B.1 in \cite{SEMP13}.
For more details, one can refer to \cite{SEMP13} and the references therein.

\begin{assumption}\label{a1}
There exists $(\bar{y}, \bar{z})\in ri(\dom (f) \times \dom (g))$ such that $\A^*\bar{y}+\B^*\bar{z}=c$.
\end{assumption}
\begin{theorem}\label{sunth}
(\cite[Theorem B.1]{SEMP13}) Suppose that the solution set of problem (\ref{tss})
is nonempty and that Assumption \ref{a1} holds. Let the sequence $\{(y^k, z^k ; x^k )\}$ be generated by iterative scheme \eqref{sadmm} from an initial point $(y^0,z^0;x^0)$.
Then, under the conditions that $\xi \in (0,(1+\sqrt{5})/2)$ and $\mathcal{T}_f$
and $\mathcal{T}_g$ be positive semi-definite, the sequence $\{(y^k,z^k;x^k)\}$ converges to a
unique limit $(\bar{y},\bar{z};\bar{x})$ with $(\bar{y},\bar{z})$ solving problem (\ref{tss}).
\end{theorem}

Next, we quickly review another type of ADMM. In order to broadening the capability of the semi-proximal ADMM (\ref{sadmm}) at the special case $\xi=1$,  Xiao, Chen \& Li \cite{XIAOMPC} introduced the following generalized semi-proximal ADMM with initial point $\tilde w^{0}=(\tilde x^{0},\tilde y^{0},\tilde z^{0})\in\X\times(\dom\, f)\times(\dom\, g)$:
\begin{equation}\label{gadmm-gu}
\left\{
\begin{array}{rl}
z^{k+1}&\displaystyle=\argmin_{z}
\big\{\L_{\sigma}(\tilde y^{k},z; \tilde x^{k})+\frac{\sigma}{2}\|z-\tilde{z}^k\|_{\T_g}^{2}\big\},
\\[2mm]
x^{k+1}&\displaystyle=\tilde x^{k}+\sigma(\A^{*} \tilde y^{k}+B^{*} z^{k+1}-c),
\\[2mm]
y^{k+1}&\displaystyle=\argmin_{y}\big\{\L_{\sigma}(y,z^{k+1}; x^{k+1})+\frac{\sigma}{2}\|y-\tilde{y}_k\|_{\T_f}^2\big\},
\\[2mm]
\tilde w^{k+1}&\displaystyle=\tilde w^{k}+\rho(w^{k+1}-\tilde w^{k}),
\end{array}\right.
\end{equation}
where $\rho\in(0,2)$ is a relaxation factor and $w^k=(x^k,y^k,z^k)$. For $\rho=1$, the above generalized ADMM scheme is exactly the classical ADMM scheme \eqref{sadmm} with $\xi=1$. When $\T_f=0$ and $\T_g=0$, the iteration (\ref{gadmm-gu}) is actually the generalized ADMM developed by Eckstein \& Bertsekas \cite{ECKSTEIN92}. For details on this equivalence, one can refer to  Chen's Ph.D. thesis \cite[Section 3.2]{CHENPHD}.

From Theorem B.1 in \cite{SEMP13} and Theorem 5.1 in \cite{XIAOMPC}, the convergence result of corresponding algorithm based on the scheme (\ref{gadmm-gu}) under Assumption \ref{a1} can be stated as follows:
\begin{theorem}\label{theorem}
(\cite[Theorem B.1]{SEMP13}, \cite[Theorem 5.1]{XIAOMPC}) Suppose that the solution set of problem (\ref{tss})
is nonempty and that Assumption \ref{a1} holds. Let the sequence $\{(y^k, z^k ; x^k )\}$ be generated by iterative scheme (\ref{gadmm-gu}) from an initial point $(\tilde x^{0},\tilde y^{0},\tilde z^{0})$. Then, under the conditions that $\rho\in (0,2)$ and that $\mathcal{T}_f$
and $\mathcal{T}_g$ be positive semi-definite, the sequence $\{(y^k,z^k;x^k)\}$ converges to a
unique limit $(\bar{y},\bar{z};\bar{x})$ with $(\bar{y},\bar{z})$ solving problem (\ref{tss}).
\end{theorem}
\section{Applying sGS-ADMM on problem (\ref{modelthree})}\label{section3}
\setcounter{equation}{0}
In this section, we quickly review the direct extend ADMM of Zhang {\itshape et al.} \cite{TILT}, and  show the applications of sGS-ADMM subsequently.
In the following, for simplicity, we may omit the superscripts $i+1$ or $i$ of variables.  This should not cause
ambiguity by referring to the context.

The Lagrangian function of (\ref{modelthree}) is given by
\begin{equation}\label{lagf}
\begin{array}{ll}
\L(X, E, \Delta\tau; Y)=&\|X\|_{*}+\lambda\| E\|_{1} \\[2mm]
&+\big\langle Y, D\circ\tau+J\Delta\tau-X-E\big\rangle,
\end{array}
\end{equation}
where $Y$ is a multiplier and $\langle \cdot, \cdot\rangle$ is the standard trace inner product. Then the dual of
(\ref{modelthree}) takes the following form
\begin{equation}\label{dualprob}
\max_{Y}\Big\{\langle Y, D\circ\tau\rangle:J^*Y=0, \ \| Y\|\leq 1, \ \big\|Y\big\|_\infty\leq\lambda \Big\},
\end{equation}
where $J^*$ is an adjoint operator, transpose in the matric case, of operator $J$; $\|\cdot\|$ is the so-called spectral norm which depends on
the largest singular value of a matrix; $\|\cdot\|_\infty$ is $\infty$-norm that is defined as
the maximum entries¡¯ magnitude of a matrix. Denote $\B_1=\{Y \ | \ \|Y\|\leq 1\}$ and $\B_2=\{Y \ | \ \|Y\|_\infty\leq\lambda\}$. The model (\ref{dualprob})
can be equivalently reformulated as
\begin{equation}\label{dualprob2}
\min_{Y}\Big\{-\langle Y, D\circ\tau\rangle:J^*Y=0, \ Y\in\B_1, Y\in\B_2 \Big\}.
\end{equation}
We say that $\bar{Y}$ is the Lagrangian multiplier of (\ref{dualprob2}) at point $(\bar{X},\Delta\bar{\tau},\bar{E})$, if it
satisfies the KKT condition:
\begin{equation}\label{kkt}
\left\{
\begin{array}{l}
-D\circ\tau-J\Delta\tau+X+E=0,\\[1mm]
J^*Y=0,\\[1mm]
0\in-X+\N_{\B_1}(Y),\\[1mm]
0\in-E+\N_{\B_2}(Y),
\end{array}
\right.
\end{equation}
where $\N_{\B_1}(Y)$ (resp. $\N_{\B_1}(Y)$) is the normal cone to $\B_1$ (resp. $\B_2$) at $Y\in\B_1$ (resp. $Y\in\B_1$).

The augmented Lagrangian function associated with (\ref{modelthree}) is defined by:
\begin{equation}\label{algf}
\begin{array}{ll}
   &\L_{\sigma}(X, E, \Delta\tau; Y)\\[2mm]
=&\|X\|_{*}+\lambda\| E\|_{1}+\big\langle Y, D\circ\tau+J\Delta\tau-X-E\big\rangle\\[2mm]
   &+\frac{\sigma}{2}\| D\circ\tau+J\Delta\tau-X-E\|^{2}_{F},
\end{array}
\end{equation}
where  $\|\cdot\|_{F}$
is the Frobenius norm, and $\sigma>0$ is a penalty parameter.
The directly extended ADMM implemented by Zhang {\itshape et al.} \cite{TILT} minimizes $ \L_{\sigma}(X, E, \Delta\tau; Y)$ firstly
with respect to $X$, later with $E$, and then with $\Delta\tau$ by fixing
other variables with their latest values. More precisely, with the
given $(X^k,E^k,\Delta\tau^k;Y^k)$, it generates the new iterate
$(X^{k+1},E^{k+1},\Delta\tau^{k+1};Y^{k+1})$ via the iterative scheme:
\begin{equation}\label{admmtilt}
\left\{
\begin{array}{lll}
  X^{k+1}&=&\argmin_{X} \L_\sigma(X, E^{k}, \Delta\tau^{k}; Y^{k}),\\[2mm]
  \Delta\tau^{k+1} &=&\argmin_{\Delta\tau} \L_\sigma(X^{k+1}, E^{k}, \Delta\tau; Y^{k}),\\[2mm]
  E^{k+1} &=&\argmin_{E} \L_\sigma(X^{k+1}, E, \Delta\tau^{k+1}; Y^{k}),\\[2mm]
  Y^{k+1}&=&Y^{k}+\xi\sigma(D\circ\tau+J\Delta\tau^{k+1}-X^{k+1}-E^{k+1}).
\end{array}
\right.
\end{equation}
Although each step of the above iteration involves solving a convex minimization problem, it was shown in \cite{TILT} that a simple
closed-form solution is permitted for each subproblem. For the completeness of this paper, we re-present the derivations for each step of (\ref{admmtilt}) by using the
new notations reported in Propositions \ref{prop1} and \ref{prop2}. Firstly, we can get for every $k=0,1,\ldots$ that
\begin{align*}
X^{k+1}&=\argmin_{X} \L_{\sigma}(X, E^{k}, \Delta\tau^{k}; Y^{k})\\
       &=\argmin_X \Big\{\|X\|_{*}\\
       &\quad+\frac{\sigma}{2}\big\|X-(D\circ\tau+J\Delta\tau^k-E^k+Y^k/\sigma)\big\|^{2}_{F}\Big\}\\
       &=\D_{1/\sigma}(D\circ\tau+J\Delta\tau^k-E^k+Y^k/\sigma),
\end{align*}
where the last equality is from Proposition \ref{prop1}. Secondly, for every $k=0,1,\ldots$, we have
\begin{align*}
  \Delta\tau^{k+1} &=\argmin_{\Delta\tau} \L_{\sigma}(X^{k+1}, E^{k}, \Delta\tau; Y^{k})\\
                     &=\argmin_{\Delta\tau} \Big\{\big\langle Y^k, J\Delta\tau\big\rangle\\
                     &\qquad+\frac{\sigma}{2}\| D\circ\tau+J\Delta\tau-X^{k+1}-E^k\|^{2}_{F}\Big\},
\end{align*}
which amounts to solving the following linear system of equations with variable $\Delta\tau$
$$
J^* Y^k/\sigma+J^*(D\circ\tau+J\Delta\tau-X^{k+1}-E^k)=0.
$$
Hence, the solution $\Delta\tau^{k+1/2}$  is given explicitly by
$$
\Delta\tau^{k+1}=-(J^* J)^{-1}[J^*(D\circ\tau-X^{k+1}-E^k)+J^* Y^k/\sigma].
$$
Thirdly, for every $k=0,1,\ldots$, we have
\begin{align*}
E^{k+1} &=\argmin_{E} \L_{\sigma}(X^{k+1}, E, \Delta\tau^{k+1}; Y^{k})\\
        &=\argmin_{E}\Big\{\lambda\| E\|_{1}\\
        &\quad+\frac{\sigma}{2}\big\| E-(D\circ\tau+J\Delta\tau^{k+1}-X^{k+1}+Y^k/\sigma)\big\|^{2}_{F}\Big\}\\
        &=\S_{\lambda/\sigma}(D\circ\tau+J\Delta\tau^{k+1}-X^{k+1}+Y^k/\sigma),
\end{align*}
where the last equality is from Proposition \ref{prop2}.

Although the direct extension of ADMM scheme indeed works empirically to produce corrected solutions, it was shown in \cite{CHENMP16} that the scheme (\ref{admmtilt}) is not necessarily convergent in theory. Ideally, we should find a convergent variant which is at least as efficient as the directly extended
ADMM (\ref{admmtilt}) in practice. We achieve this goal by adopting the clever sGS technique developed recently by Li, Sun \& Toh \cite{XDLIMP}. 

Based on the sGS technique \cite{XDLIMP}, we view $X$ as one group and $(\Delta\tau, E)$ as another, and present the following iterative framework: Given $(X^k,E^k,\Delta\tau^k)$, we compute the next iteration with $X\rightarrow\Delta\tau\rightarrow E\rightarrow \Delta\tau$ instead of the usual $X\rightarrow\Delta\tau\rightarrow E$ Gauss-Seidel fashion (\ref{admmtilt}), which can be reduced to the following iterative scheme:
\begin{equation}\label{admmsgs}
\left\{
\begin{array}{lll}
  X^{k+1}&=&\argmin_{X} \L_{\sigma}(X, E^{k}, \Delta\tau^{k}; Y^{k}),\\[2mm]
  \Delta\tau^{k+1/2} &=&\argmin_{\Delta\tau} \L_{\sigma}(X^{k+1}, E^{k}, \Delta\tau; Y^{k}),\\[2mm]
  E^{k+1} &=&\argmin_{E} \L_{\sigma}(X^{k+1}, E, \Delta\tau^{k+1/2}; Y^{k}),\\[2mm]
  \Delta\tau^{k+1} &=&\argmin_{\Delta\tau} \L_{\sigma}(X^{k+1}, E^{k+1}, \Delta\tau; Y^{k}),\\[2mm]
  Y^{k+1}&=&Y^{k}+\xi\sigma(D\circ\tau+J\Delta\tau^{k+1}-X^{k+1}-E^{k+1}).
\end{array}
\right.
\end{equation}
Note that the difference between the sGS based iterative scheme (\ref{admmsgs}) and the directly extended ADMM (\ref{admmtilt}) is that we perform an
extra preparation step to compute $\Delta\tau^{k+1/2}$ and then compute $E^{k+1}$. As we can see from the previous statement that the extra step can be done at moderate cost, so that the iterative process can be performed cheaply.

With the descriptions on how the subproblems in (\ref{admmtilt}) are solved as in \cite{TILT}, we now present the sGS-ADMM method in \cite{XDLIMP} for solving (\ref{modelthree}). 

\begin{framed}
\noindent
{\bf Algorithm: (sGS-ADMM)}
\vskip 1.0mm \hrule \vskip 1mm
\noindent
1. Initialization: Input deformed and corrupted image $D\in\mathbb{R}^{m\times n}$ and its Jabobian $J$ against deformation $\tau$. Choose constants $\lambda>0$, $\sigma>0$, and $\xi\in(0,(1+\sqrt{5})/2)$.
Choose starting point $(X^0,\Delta\tau^0,E^0,Y^0)$.\\
2. {\bf while.} ``not converge", {\bf do}\\[2mm]
3. \  $X^{k + 1} = \D_{1/\sigma}(D\circ\tau+J\Delta\tau^k-E^k+Y^k/\sigma)$;\\[2mm]
4. \ $\Delta\tau^{k+1/2}=-(J^* J)^{-1}[J^* (D\circ\tau-X^{k+1}-E^k)+J^*Y^k/\sigma]$; \\[2mm]
5. \ $E^{k + 1}=\S_{\lambda/\sigma}(D\circ\tau+J\Delta\tau^{k+1/2}-X^{k+1}+Y^k/\sigma)$;\\[2mm]
6. \  $\Delta\tau^{k+1}=-(J^* J)^{-1}[J^*(D\circ\tau-X^{k+1}-E^{k+1})+J^*Y^k/\sigma]$; \\[2mm]
7. \  $Y^{k+1}=Y^{k}+\xi\sigma(D\circ\tau+J\Delta\tau^{k+1}-X^{k+1}-E^{k+1})$;\\[2mm]
8. {\bf end while.}\\
9. Output: Solution $(X,\Delta \tau, E)$ of problem (\ref{modelthree}).
\end{framed}

The remaining task is to establish the convergence result of sGS-ADMM by using the sGS decomposition theorem of Li, Sun \& Toh \cite{SGSTH} to associate it
with the semi-proximal ADMM (\ref{sadmm}). The relationship between both methods are reported in the lemma below. 
\begin{lemma}\label{propositon1}
For any $k\geq0$, the $E$- and $\Delta\tau$-subproblems in (\ref{admmsgs}) can be summarized as the following compact form:
\begin{equation}\label{propo2}
\begin{array}{ll}
   &(E^{k+1},\Delta\tau^{k+1}) \\
  =& \argmin_{E,\Delta\tau}\Big\{\L_{\sigma}(X^{k+1}, E, \Delta\tau; Y^{k})\\
  &\qquad\qquad\quad\quad+\frac{\sigma}{2}\Big\|
  \left(
   \begin{array}{c}
   E\\
  \Delta\tau\\
   \end{array}
   \right)-
   \left(
   \begin{array}{c}
   E^{k}\\
  \Delta\tau^{k}\\
   \end{array}
   \right)
   \Big\|^{2}_{\T}\Big\}.
\end{array}
\end{equation}
\end{lemma}
\begin{proof}
To prove the lemma, it is sufficient to note that the matrix for the quadratic term associated with $(E,\Delta\tau)$ is given by
\begin{equation*}
  \H=\left(
   \begin{array}{cc}
   I&-J\\
   -J^*&J^{*}J\\
   \end{array}
   \right)=\Q+\M+\M^{*},
\end{equation*}
where
$$
\M=
  \left(
   \begin{array}{cc}
   0&-J\\
   0&0\\
   \end{array}
   \right),
  \
 \Q=
  \left(
   \begin{array}{cc}
   I&0\\
   0&J^{*}J\\
   \end{array}
   \right),
$$
and
$$
 \M^{*}=
  \left(
   \begin{array}{cc}
   0&0\\
   -J^{*}&0\\
   \end{array}
   \right).
$$
By directly applying the sGS decomposition theorem in \cite{SGSTH}, and setting
 \begin{equation}\label{xiet}
   \T=\M\Q^{-1}\M^{*}=\left(
   \begin{array}{cc}
   J(J^{*}J)^{-1}J&0\\
   0&0\\
   \end{array}
   \right),
 \end{equation}
the required conclusion follows.
\end{proof}

Based on the result, we can rewrite the iterative scheme (\ref{admmsgs}) as

\begin{equation}\label{admmsgs2}
\left\{
\begin{array}{cl}
  X^{k+1}&=\argmin_{X} L(X, E^{k}, \Delta\tau^{k}; Y^{k}),\\[2mm]
  (E^{k+1},\Delta\tau^{k+1}) &= \argmin_{E,\Delta\tau}\Big\{\L_{\sigma}(X^{k+1}, E, \Delta\tau; Y^{k})\\[2mm]
  & \quad +\frac{\sigma}{2}\Big\|
  \left(
   \begin{array}{c}
   E\\
  \Delta\tau\\
   \end{array}
   \right)-
   \left(
   \begin{array}{c}
   E^{k}\\
  \Delta\tau^{k}\\[2mm]
   \end{array}
   \right)
   \Big\|^{2}_{\T}\Big\},\\[2mm]
  Y^{k+1}&=Y^{k}+\xi\sigma(D\circ\tau+J\Delta\tau^{k+1}\\[2mm]
         &\qquad\qquad\qquad-X^{k+1}-E^{k+1}),
\end{array}
\right.
\end{equation}
which reduces to the two-block semi-proximal ADMM (\ref{sadmm}). Note that the main idea of the sGS decomposition theorem \cite{SGSTH} for deriving the convergence of sGS-ADMM by showing that it is equivalent to two-block ADMM with a special semi-proximal term $\T$. This equivalence is very important because the convergence can be easily followed by using 
the known convergence result of Fazel et al. \cite{SEMP13}.
To conclude this section, we present the convergence result of sGS-ADMM for solving (\ref{modelthree}).

\begin{theorem}\label{theo1}
(\cite[Theorem B.1]{SEMP13}) Let the sequence $\{(X^k,\Delta\tau^k,E^k,Y^k)\}$ be generated by Algorithm sGS-ADMM with $\xi\in(0,(1+\sqrt{5})/2)$, then it converges to the accumulation point $(\bar{X},\Delta\bar{\tau}, \bar{E},\bar{Y})$ such that  $(\bar{X},\Delta\bar{\tau}, \bar{E})$ is the solution of the problem (\ref{modelthree}).
\end{theorem}

\section{An sGS based generalized ADMM}\label{theory}
\setcounter{equation}{0}

This section is devoted to introducing a generalized variant of sGS-ADMM for solving probelm (\ref{modelthree}). Again, variable $X$
is viewed as one group and $(\Delta\tau, E)$ as another, and sGS technique with order $\Delta\tau\rightarrow E\rightarrow \Delta\tau$
is used in this group. For convenience, we denote $\Omega=(X,\Delta\tau,E,Y)$. The aforementioned sGS-ADMM will make a very small
modification, i.e., adding an extra relaxation step, which amounts to the algorithm below with an initial porint $\tilde{\Omega}^0=(\tilde{X}^0,\Delta\tilde{\tau}^0,\tilde{E}^0,\tilde{Y}^0)$.

\begin{framed}
\noindent {\bf Algorithm: (sGS-ADMM\_G)}
\vskip 1.0mm \hrule \vskip 1mm
\noindent
1. Initialization: Input deformed and corrupted image $D\in\mathbb{R}^{m\times n}$ and its Jabobian $J$ against deformation $\tau$. Choose constants $\lambda>0$, $\sigma>0$, and $\rho\in(0,2)$. Choose starting point $(\tilde X^0,\tilde E^0,\tilde Y^0)$\\
2. {\bf while.} ``not converge", {\bf do}\\[2mm]
3. \  $\Delta\tau^{k+1/2}=-(J^* J)^{-1}[J^*(D\circ\tau-\tilde X^{k}-\tilde E^k)+ J^*\tilde Y^k/\sigma]$; \\[2mm]
4. \  $E^{k + 1}= \S_{\lambda/\sigma}(D\circ\tau+J\Delta\tau^{k+1/2}-\tilde X^{k}+\tilde Y^k/\sigma)$;\\[2mm]
5. \  $\Delta\tau^{k+1}=-(J^* J)^{-1}[J^*(D\circ\tau-\tilde X^{k}-E^{k+1})+J^*\tilde Y^k/\sigma]$; \\[2mm]
6. \  $Y^{k+1}=Y^{k}+\sigma(D\circ\tau+J\Delta\tau^{k+1}-\tilde X^{k}-E^{k+1})$;\\[2mm]
7. \  $X^{k + 1} = \D_{1/\sigma}(D\circ\tau+J\Delta\tau^{k+1}-E^{k+1}+Y^{k+1}/\sigma)$;\\[2mm]
8. \  $\tilde \Omega^{k+1}=\tilde \Omega^{k}+\rho(\Omega^{k+1}-\tilde \Omega^{k});$\\[2mm]
9. {\bf end while.}\\
10. Output: Solution $(X,\Delta \tau, E)$ of problem (\ref{modelthree}).
\end{framed}

Analogously, based on the novel sGS decomposition theorem of Li, Sun \& Toh \cite{SGSTH}, we can reformulated sGS-ADMM\_G as the following framework: \begin{equation}\label{admmsgs-g2}
\left\{
\begin{array}{l}
  (E^{k+1},\Delta\tau^{k+1}) = \argmin_{E,\Delta\tau}\Big\{\L_{\sigma}(\tilde X^{k}, E, \Delta\tau; \tilde Y^{k})\\[2mm]
   \qquad\qquad\qquad+\frac{\sigma}{2}\Big\|
  \left(
   \begin{array}{c}
   E\\
  \Delta\tau\\
   \end{array}
   \right)-
   \left(
   \begin{array}{c}
   \tilde E^{k}\\
  \tilde \Delta\tau^{k}\\
   \end{array}
   \right)
   \Big\|^{2}_{\T}\Big\},\\  [2mm]
  Y^{k+1}=\tilde Y^{k}+\sigma(D\circ\tau+J\Delta\tau^{k+1}-\tilde X^{k}-E^{k+1}),\\[2mm]
  X^{k+1}=\argmin_{X} \L_{\sigma}(X, E^{k+1}, \Delta\tau^{k+1}; Y^{k+1}),  \\[2mm]
  \tilde \Omega^{k+1}=\tilde \Omega^{k}+\rho(\Omega^{k+1}-\tilde \Omega^{k}),
\end{array}
\right.
\end{equation}
where $\T$ is defined in (\ref{xiet}).Therefore, according to Theorem B.1 in \cite{SEMP13} and Theorem 5.1 in \cite{XIAOMPC}, the convergence result of Algorithm sGS-GADMM\_G can be listed.

\begin{theorem}\label{theo2}
(\cite[Theorem B.1]{SEMP13}, \cite[Theorem 5.1]{XIAOMPC}) Let the sequence $\{(X^k,\Delta\tau^k,E^k,Y^k)\}$ be generated by Algorithm sGS-ADMM\_G with $\rho\in(0,2)$, then it is automatically well-defined and converges to the accumulation point $(\bar{X},\Delta\bar{\tau}, \bar{E},\bar{Y})$ such that  $(\bar{X},\Delta\bar{\tau}, \bar{E})$ is the solution of the problem (\ref{modelthree}).
\end{theorem}

\section{Numerical Experiments}\label{numersec}
In this section, we construct a series of numerical experiments by using deformations contained real images to evaluate the practical performance of
algorithms sGS-ADMM and sGS-ADMM\_G. In the mean time, we also test against the directly extended ADMM approach (named TILT) to further illustrate the efficiency and
robustness of sGS-ADMM and sGS-ADMM\_G. The Matlab package for the algorithm TILT is available at the link \url{http://perception.csl.illinois.edu/matrix-rank/tilt.html}. We mention that all these algorithms are tested by running {\sc Matlab} on a LENOVO with one Intel Core i5-5200U Processor (24 Cores, 2.2 to 2.19 GHz) and 8 GB  RAM.

In order to make both algorithms well-defined on the original model (\ref{modeltwo}), we quickly review some specific implementation details reported by Zhang {\itshape et al.} \cite{TILT}, which are also used in both algorithms. Firstly, before starting the iterative process, the intensity of the image is normalized as $D\circ \tau:=D\circ\tau/\|D\circ\tau\|_F$ because the low-rank texture is invariant
with respect to scaling in the pixel values. Secondly, a set of linear constraints is added to eliminate the scaling and translation ambiguities in the solution, e.g.,
for affine transformations,  a constraint $A_t\Delta\tau=0$ (liner operator $A_t$ is known) is added to ensure that the center of the initial rectangular region remain fixed before and after the transformation.
Finally, to increase the range of deformation, TILT employed a branch-and-bound scheme, e.g., for affine case,  the affine transformation can be parameterized as $A(\theta,t)$, then TILT computes the best rotation angle $\theta$ and used it as an initialization to search for the skew parameter $t$. It was shown that these reviewed implementations improved the range of convergence of TILT potentially. Since the main contribution of our paper lies in employing  the sGS-ADMM algorithm to solve the problem (\ref{modelthree}), hence, in the following each experiment, we only consider the single affine transformation and use the branch-and-bound scheme for convenience.

We perform two classes of numerical experiments. In the first class, we evaluate the practical performance of sGS-ADMM and sGS-ADMM\_G on some natural images belong to various categories, while in the second class, we test against algorithm TILT on some representative synthetic and realistic low-rank patterns to examine both algorithms'
performance. In the considered model, we choose parameter as $\lambda=1/\sqrt{m}$ and set the initial estimation as $\tau^{(0)}=0$. The iterative processes of each algorithm start at zero, i.e., $(E^0,\Delta\tau^0,Y^0,X^0)=(\tilde{E}^0,\Delta\tilde{\tau},\tilde{Y}^0,\tilde{X}^0)=0$. Moreover, we choose $\xi=1.618$ in sGS-ADMM and $\rho=1.8$ in sGS-ADMM\_G since both values are used frequently in algorithms' designing for solving various optimization problems.

Based on the optimality condition (\ref{kkt}),
we measure the accuracy of a computed candidate solution $(E,\Delta\tau,X;Y)$ for (\ref{modelthree}) and its dual (\ref{dualprob2}) via
$$
\eta=\max\{\eta_P,\eta_D,\eta_X,\eta_E \},
$$
where
\begin{equation}\label{stop3}
\begin{array}{l}
\eta_P=\max\Big\{\frac{\|D\circ\tau+J\Delta\tau-X-E\|_F}{\|D\circ\tau\|_F},A_t\Delta\tau\Big\},\\[2mm]
\eta_D=\|J^*Y\|,\\[2mm]
\eta_X=\frac{\|Y-\Pi_{\B_1}(Y+X)\|_F}{1+\|Y\|_F+\|X\|_F},\\[2mm]
\eta_E=\frac{\|Y-\Pi_{\B_2}(Y+E)\|_F}{1+\|Y\|_F+\|E\|_F},
\end{array}
\end{equation}
where $\Pi_{\B}(\cdot)$ is the metric projection onto $\B$ under the Frobenius norm. According to the adjustment strategy in \cite{LEE}, we initialize the penalty parameter as $\sigma=1/\|D\circ\tau^{(i)}\|_F$, and increase it frequently with $\sigma=1.25\sigma$ if $\eta_P/\eta_D\geq 5$ in each test, and decrease it with $\sigma=0.8\sigma$ if $\eta_P/\eta_D\leq 1/5$. All the algorithms are terminated if $\eta<10^{-3}$, or the maximum iteration number $1,000$ is achieved.
Besides, others the parameters' in TILT are set with default values for comparison in a fair way. While the final $\Delta\bar{\tau}$ is achieved based on the current $\tau^{(i)}$, we then set $\tau^{(i+1)}:=\tau^{(i)}+\Delta\bar{\tau}$ and solve problem (\ref{modelthree}) immediately once again. As in \cite{TILT}, the external loops proceed repeatedly when the absolute values between two successive rounds are small enough, i.e., $|f^{(i+1)}-f^{(i)}|\leq 10^{-4}$, where $f^{(i)}=\| X\|_{*}+\lambda\|E\|_{1}$  at the $i$-th external loop.

In the first test, we visually examine the practical performance of sGS-ADMM and sGS-ADMM\_G.  As Figure \ref{fig1} shows, the results report the original input images and the rectified textures returned by each algorithm. As can be observed from the last two columns that, both algorithms correctly recovered the local geometry for all the textual images from green windows located inside, which in turn demonstrates the practical efficiency of both algorithms.

\begin{figure}[htbp]
\centering
  \includegraphics[scale=0.19]{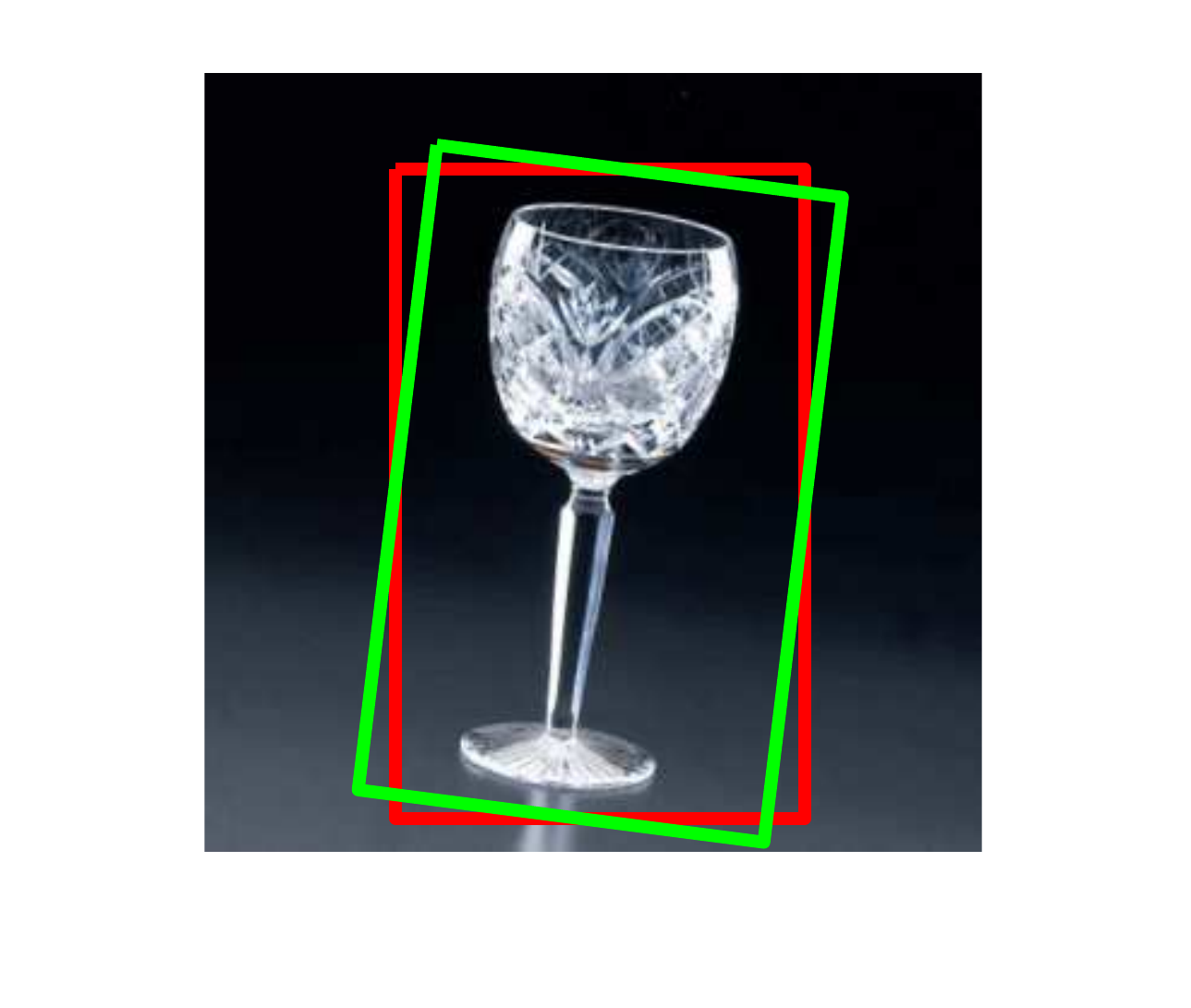}\hspace{-0.4cm}
  \includegraphics[scale=0.17]{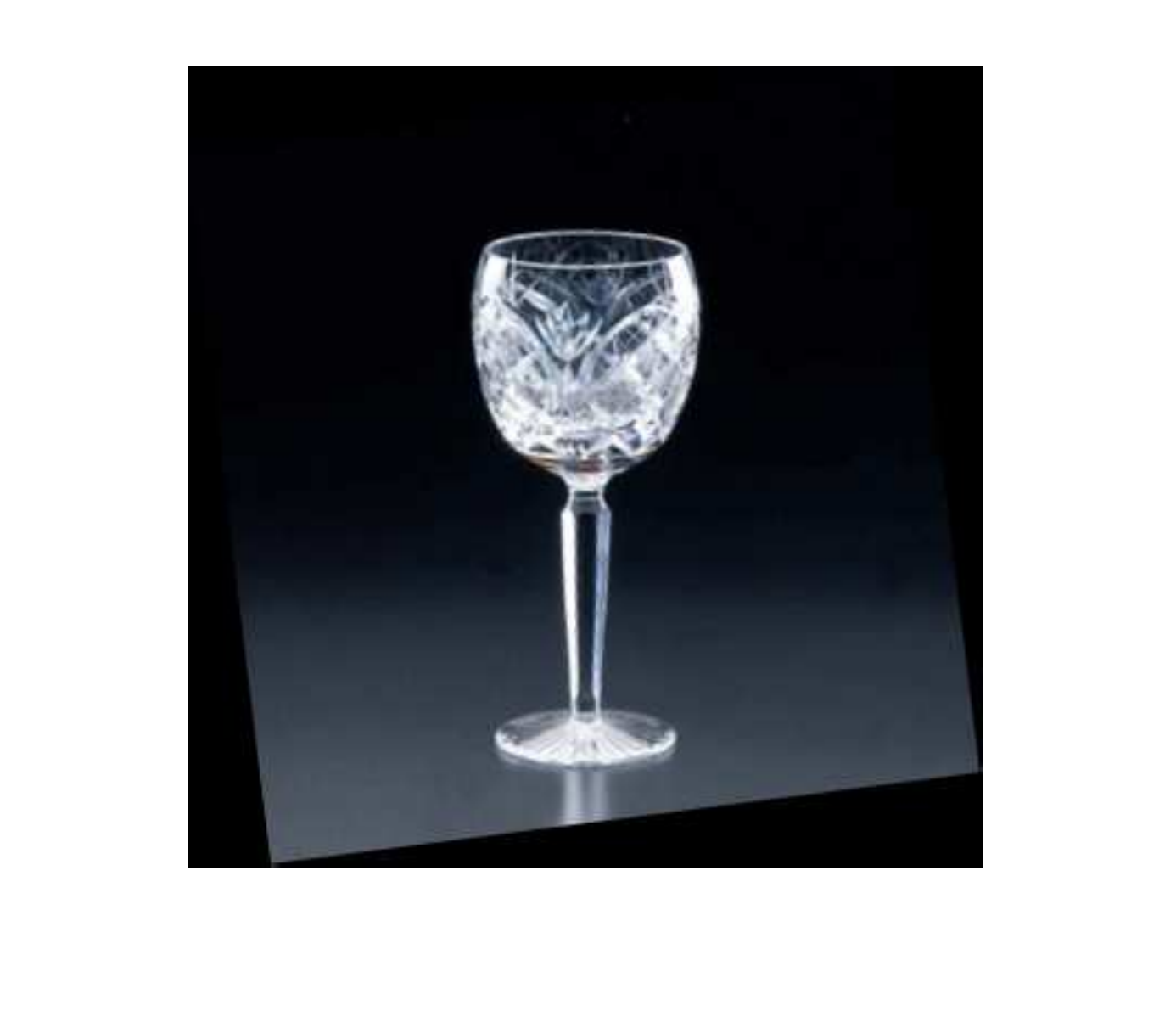}\hspace{-0.4cm}
  \includegraphics[scale=0.22]{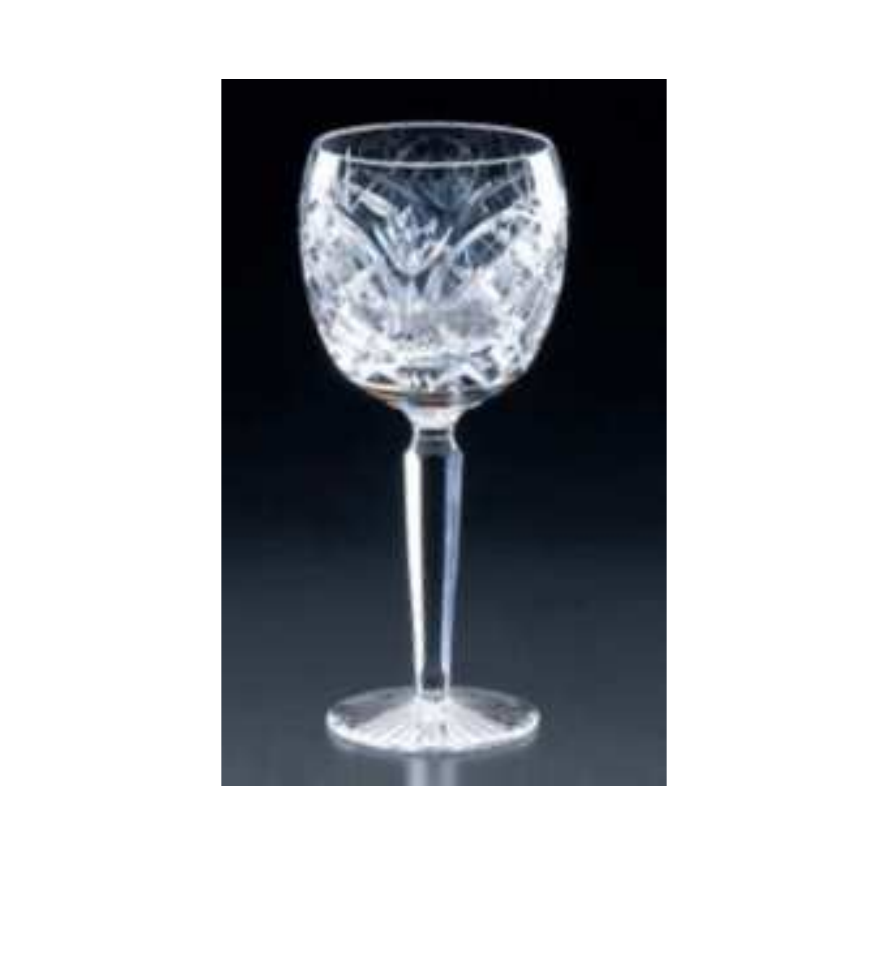}\hspace{-0.4cm}
  \includegraphics[scale=0.22]{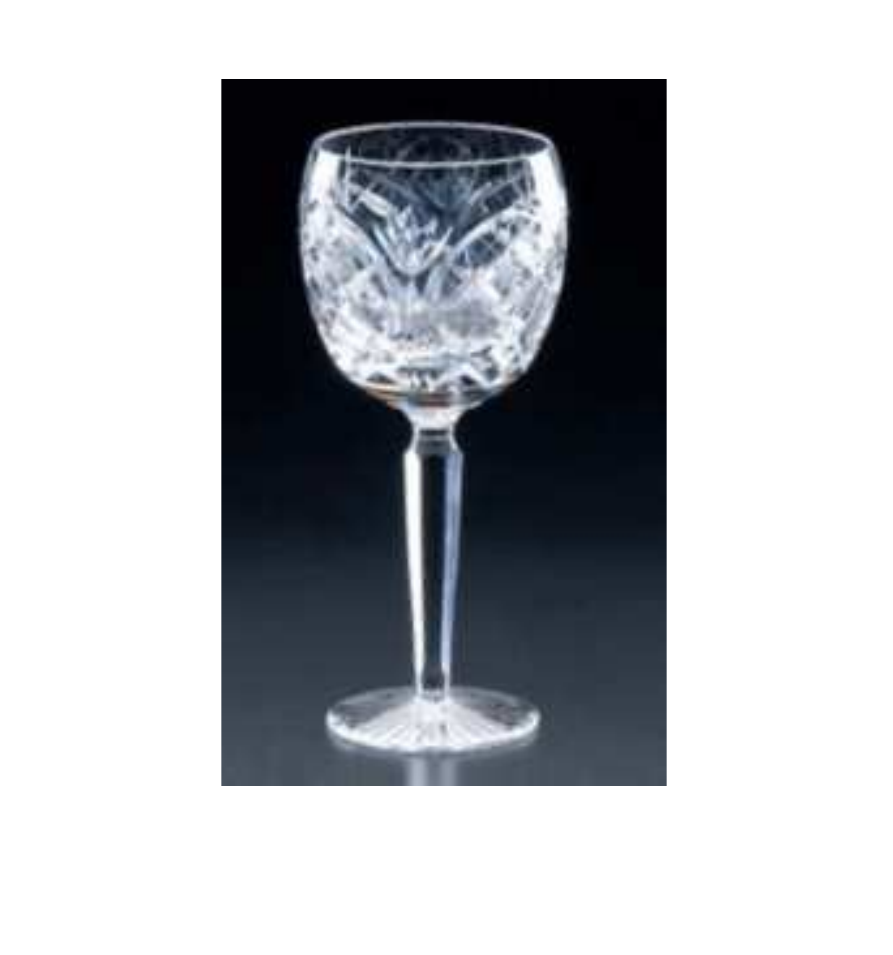}\\
  \includegraphics[scale=0.125]{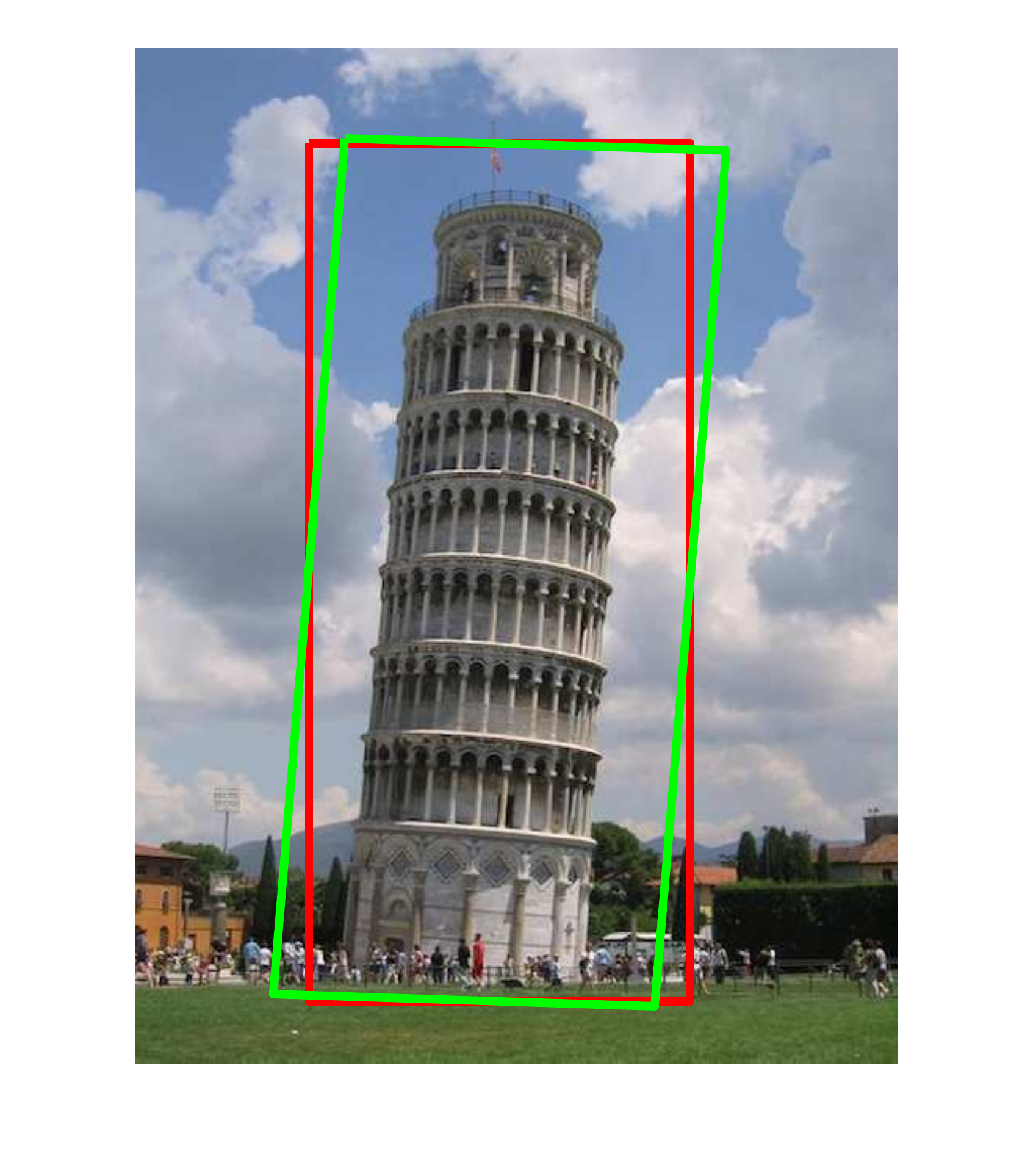}\hspace{-0.18cm}
  \includegraphics[scale=0.12] {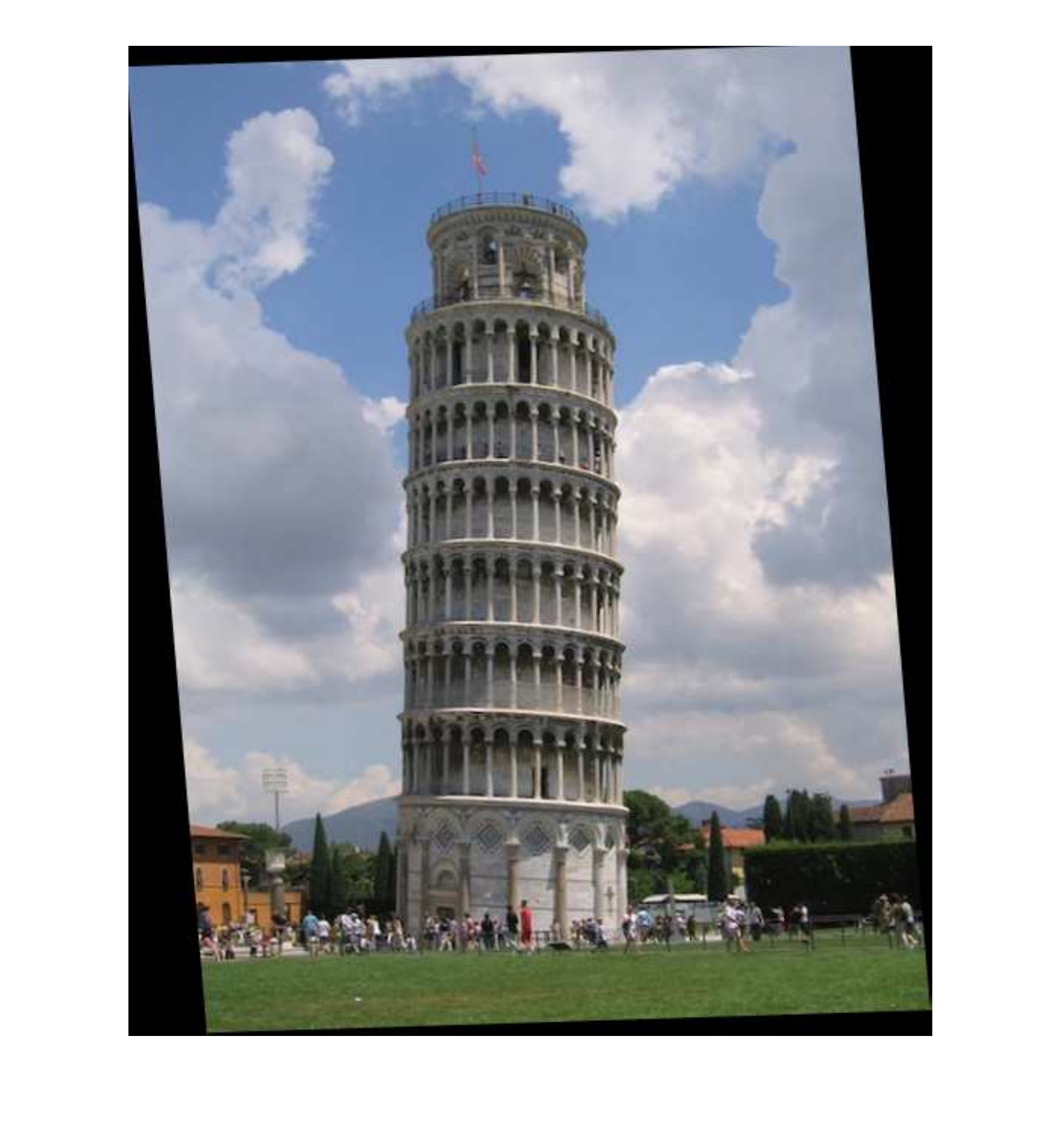}\hspace{-0.22cm}
  \includegraphics[scale=0.145]{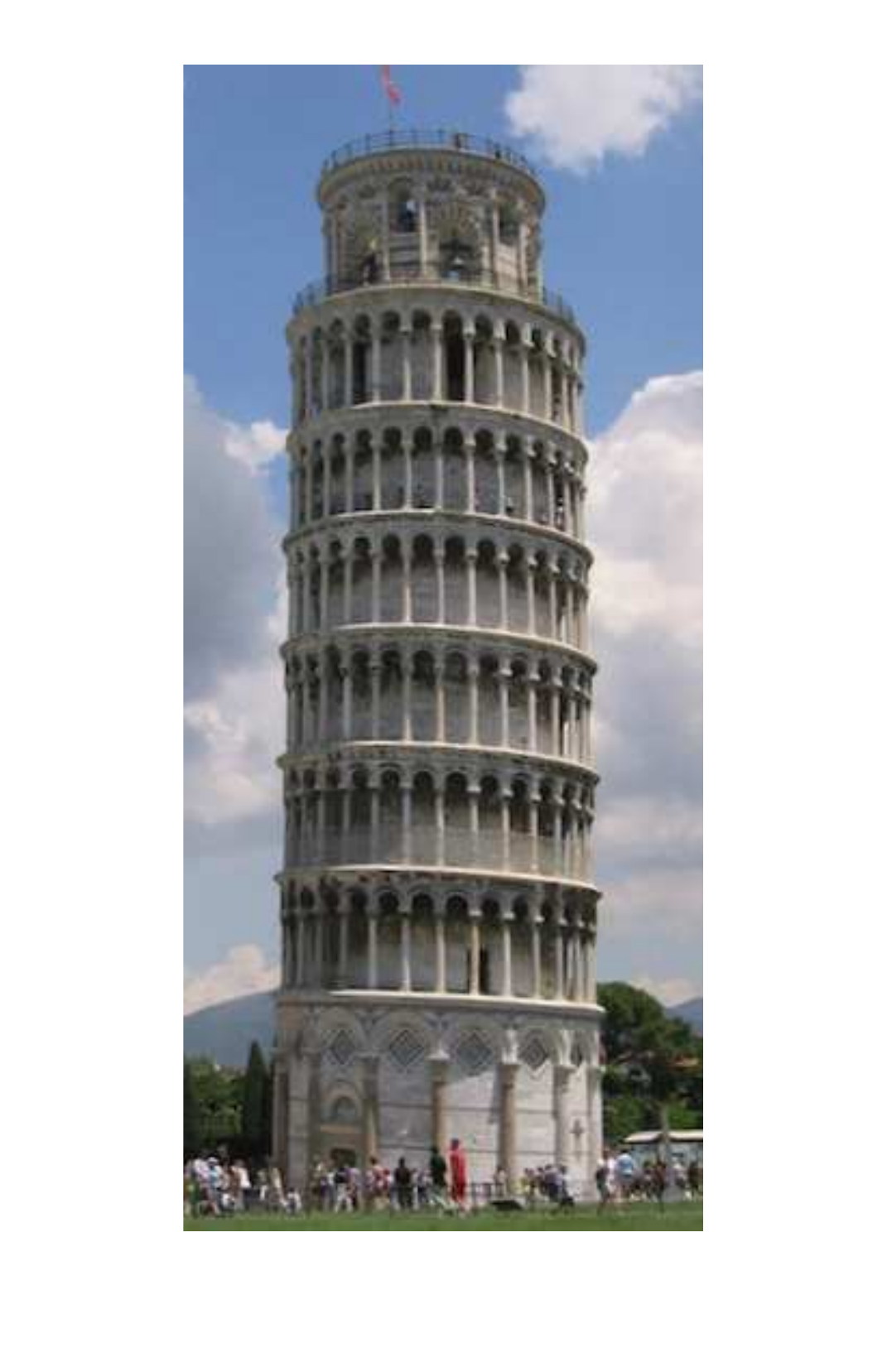}\hspace{-0.22cm}
  \includegraphics[scale=0.145]{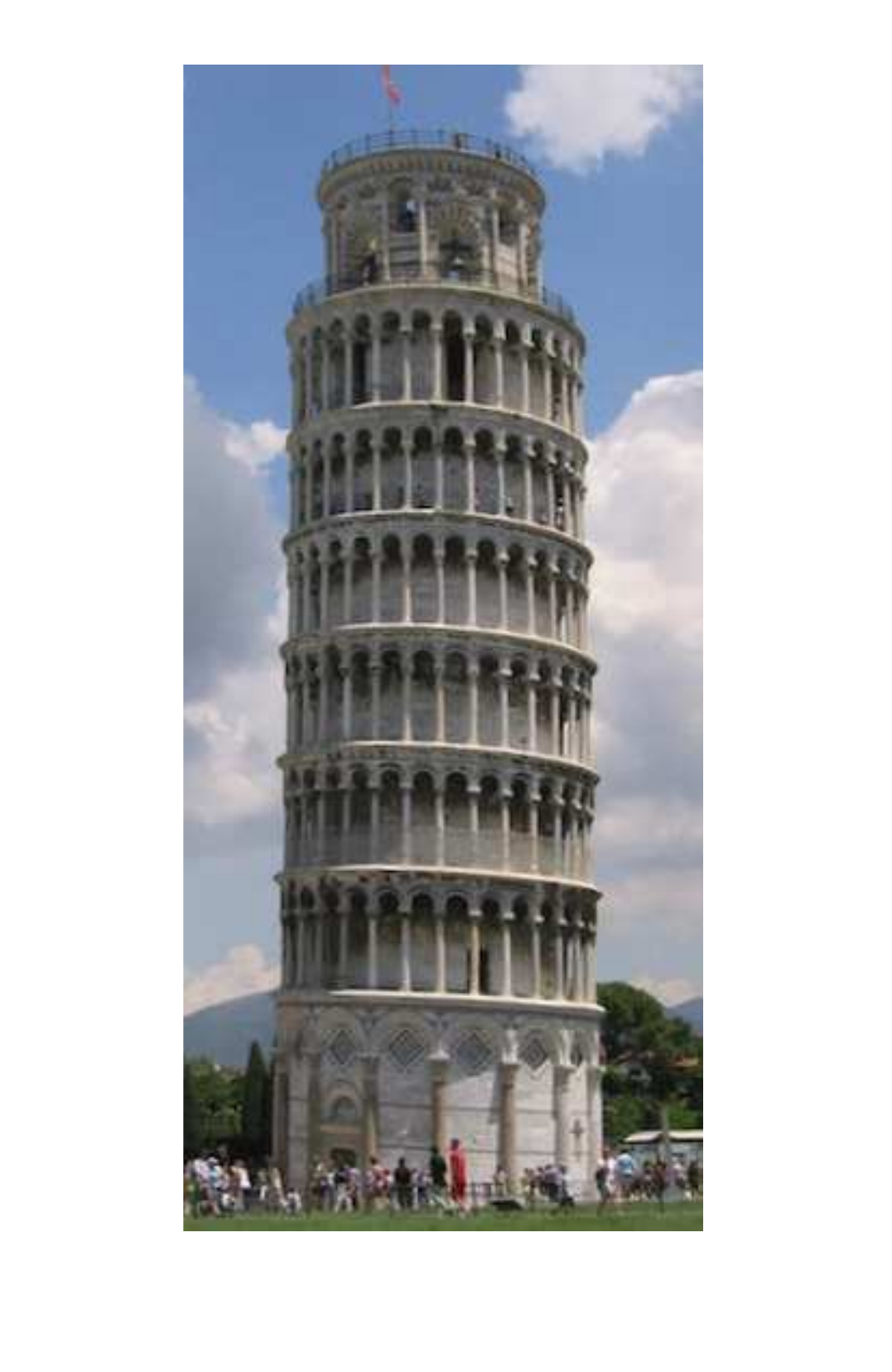}\\
  \includegraphics[scale=0.15]{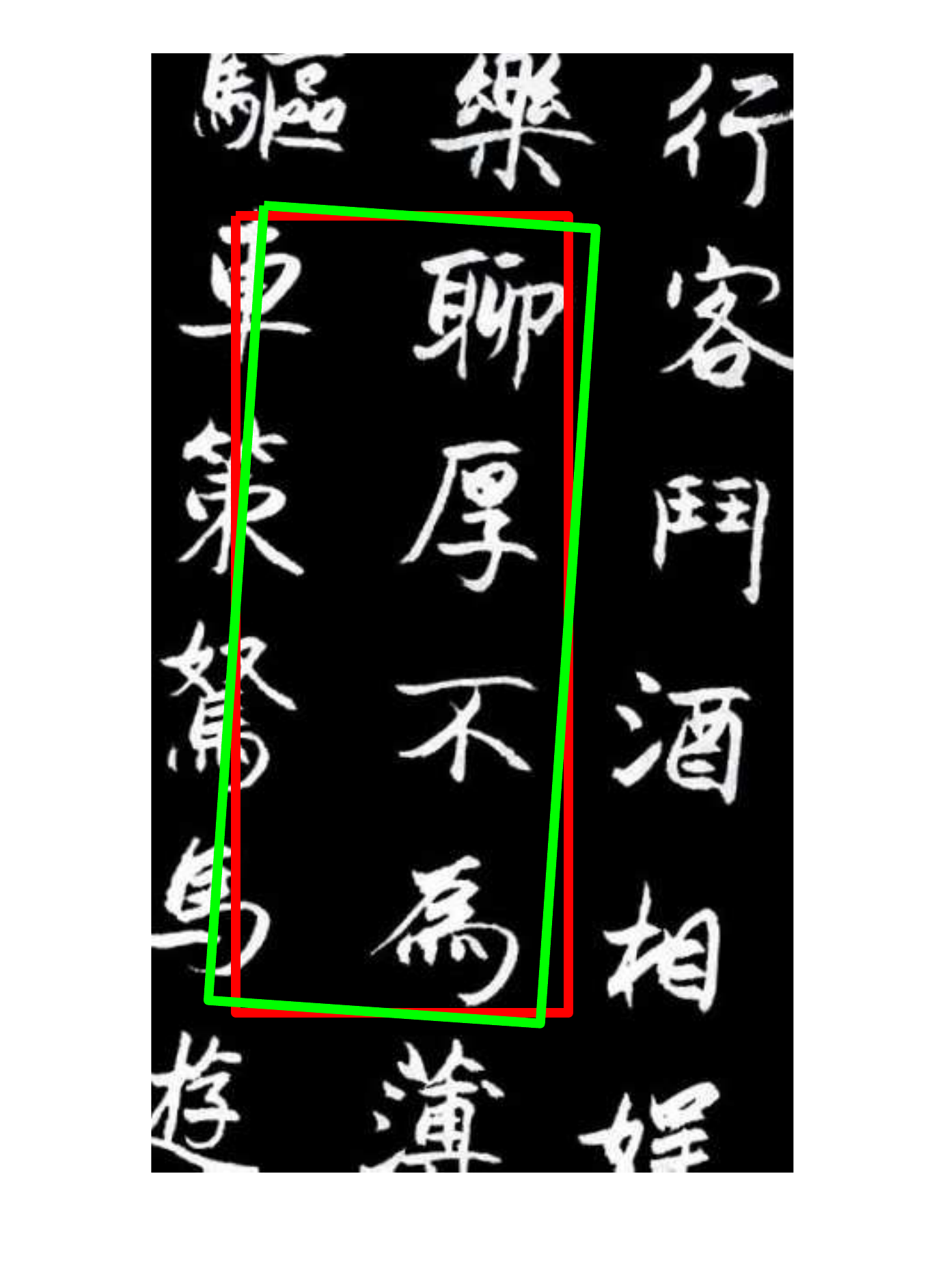}\hspace{-0.2cm}
  \includegraphics[scale=0.14]{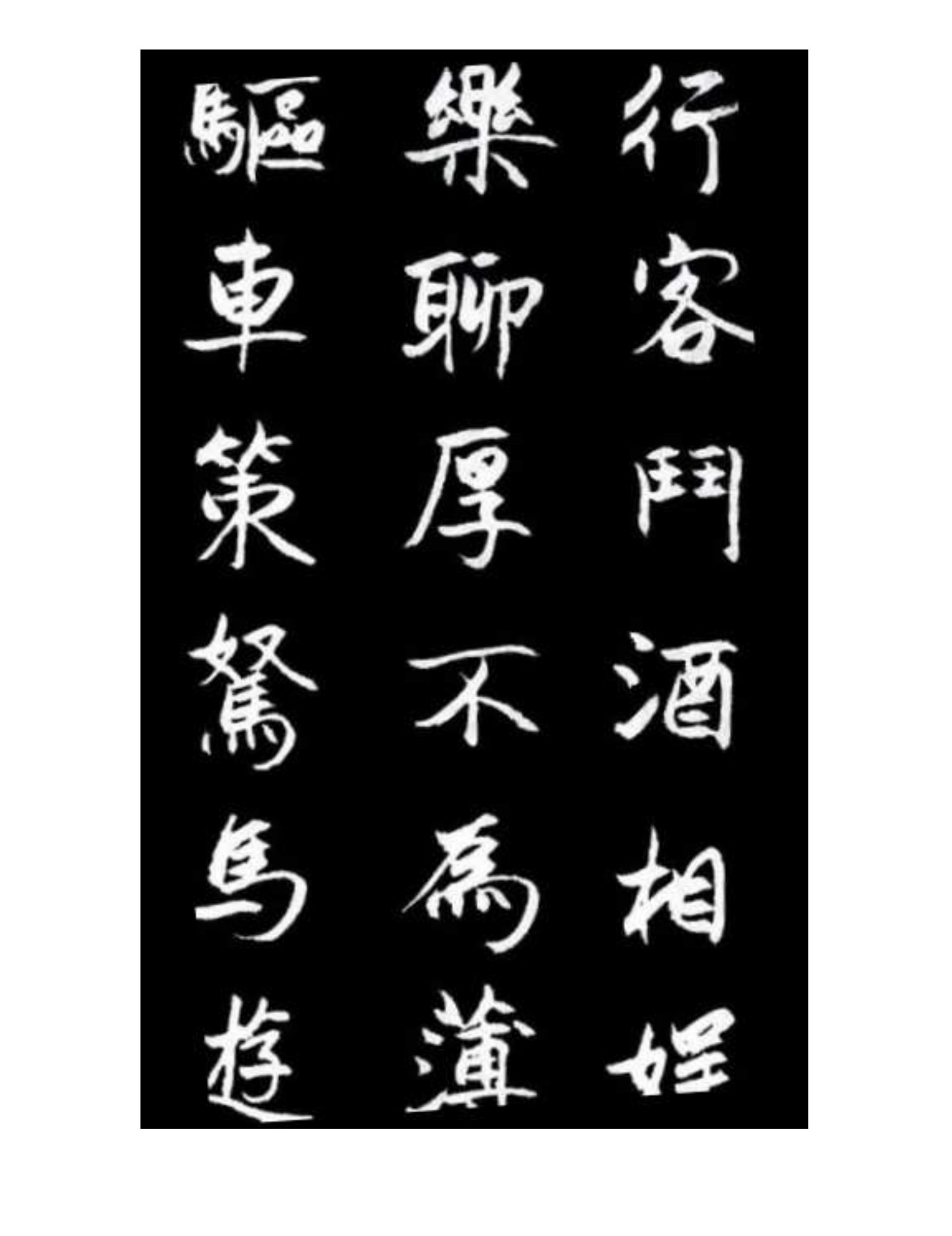}\hspace{-0.2cm}
  \includegraphics[scale=0.2] {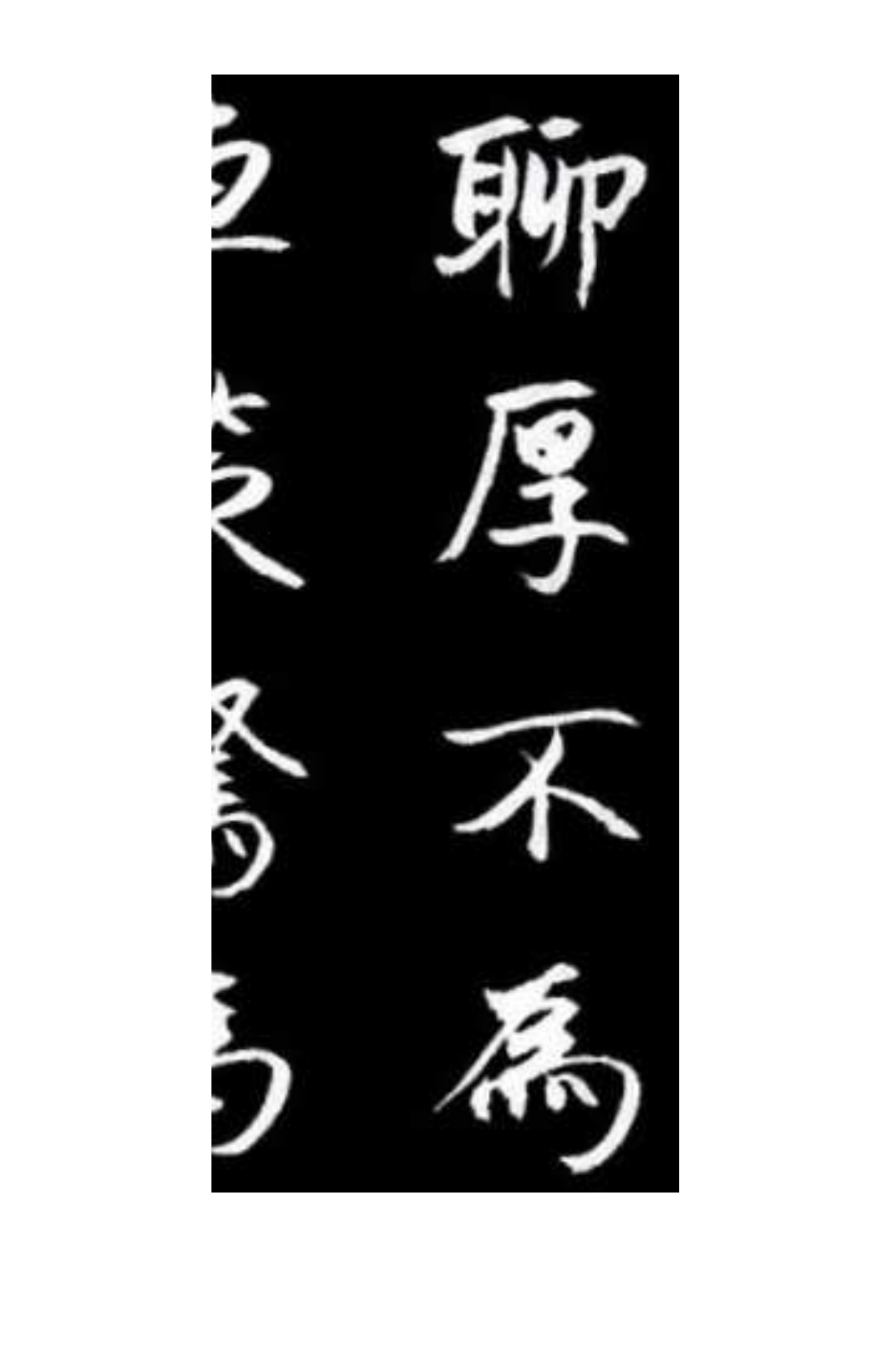}\hspace{-0.2cm}
  \includegraphics[scale=0.2] {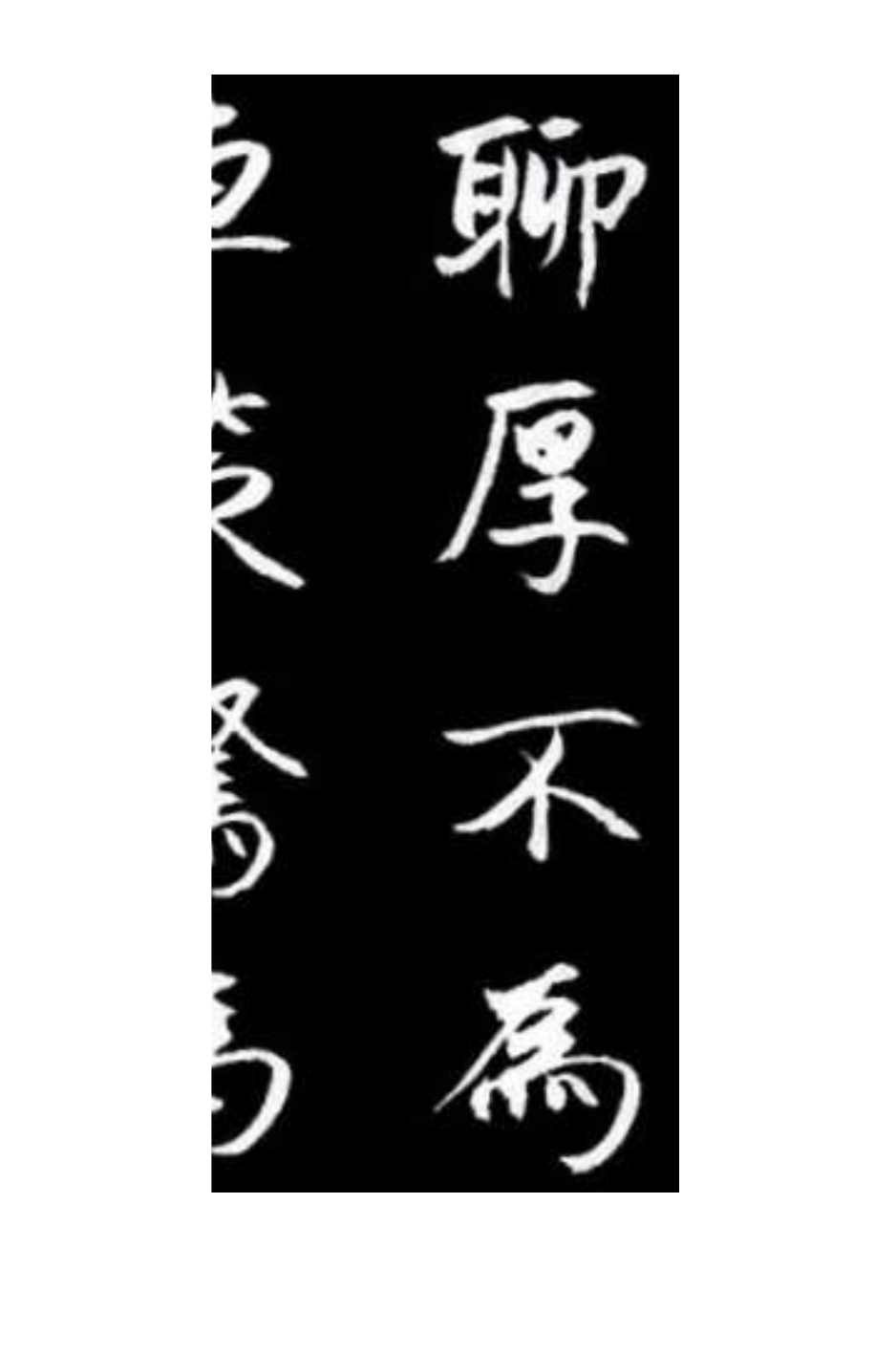}\\
  \includegraphics[scale=0.17]{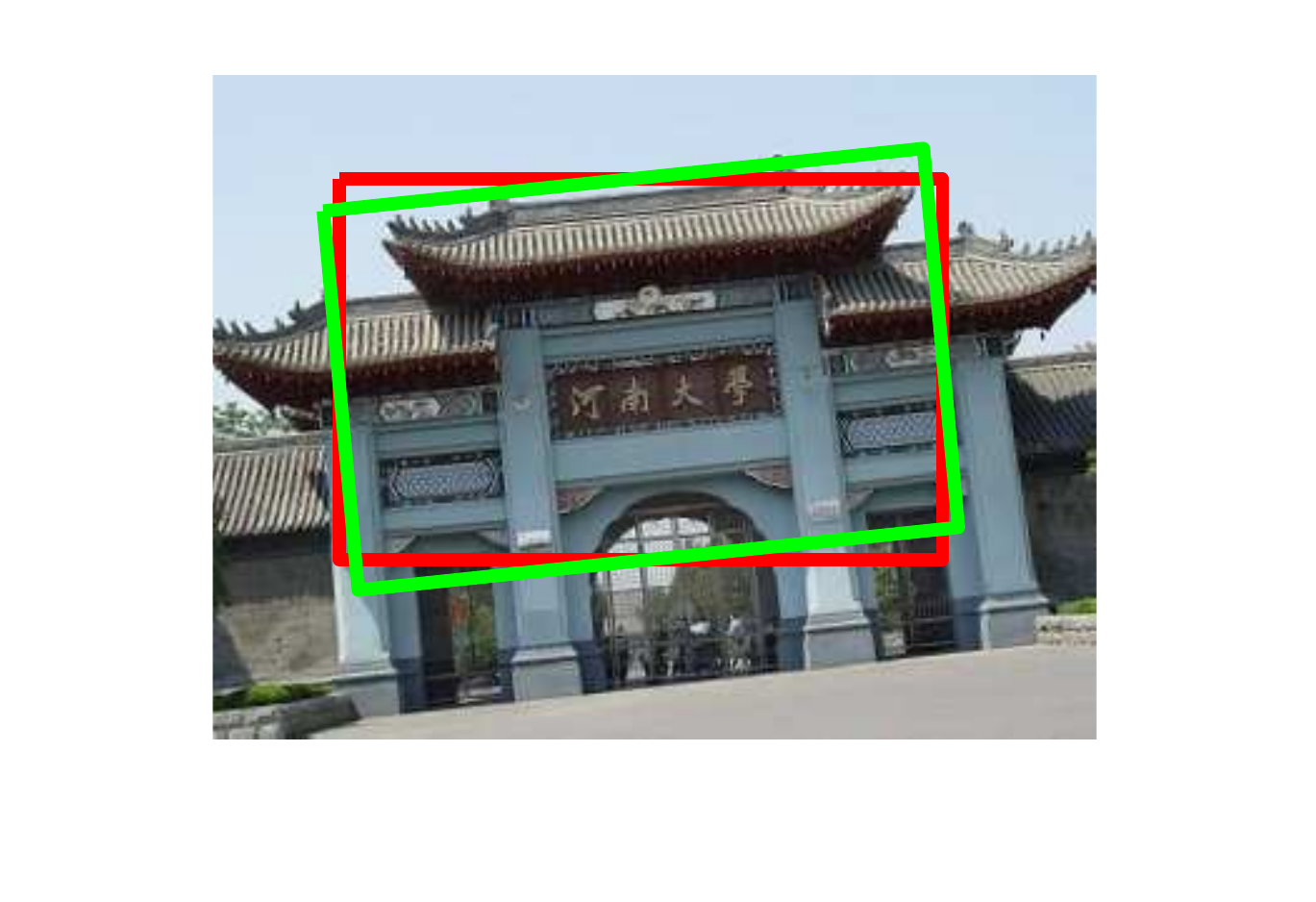}\hspace{-0.4cm}
  \includegraphics[scale=0.14]{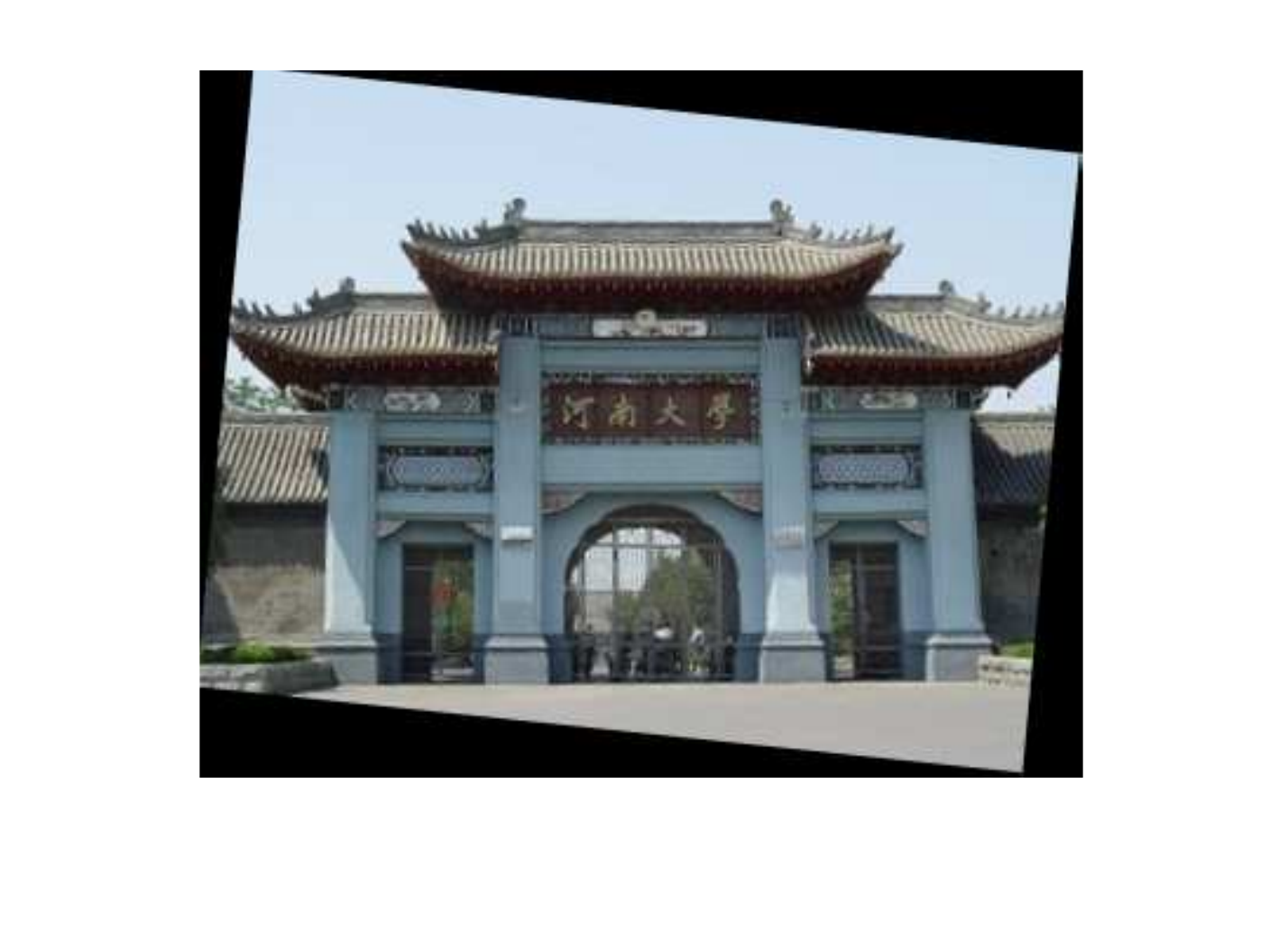}\hspace{-0.7cm}
  \includegraphics[scale=0.24]{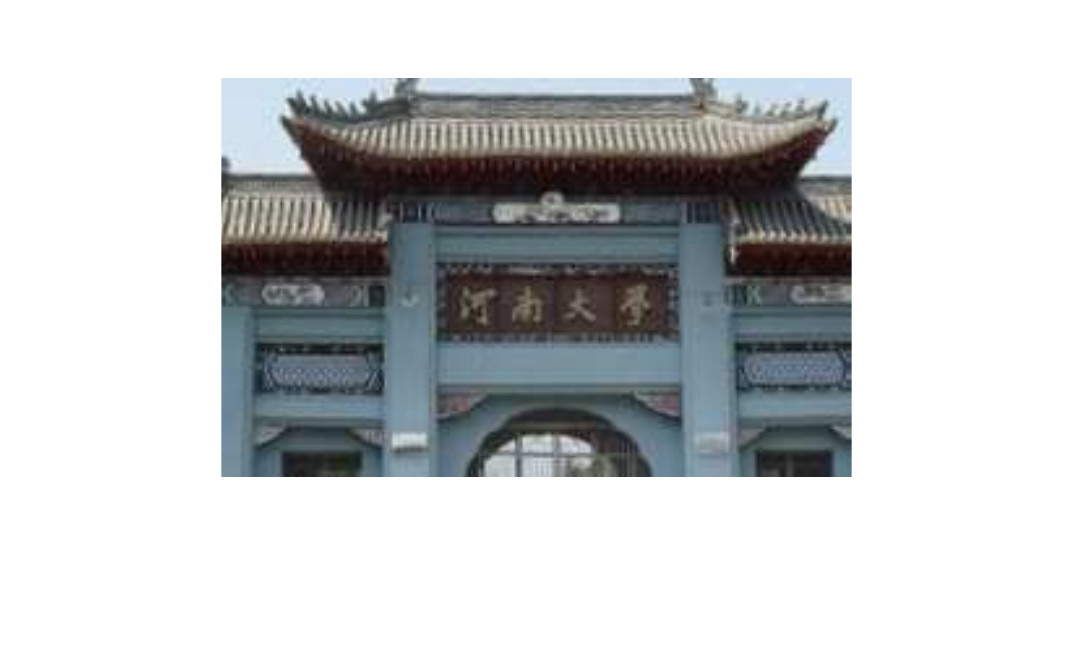}\hspace{-0.7cm}
  \includegraphics[scale=0.24]{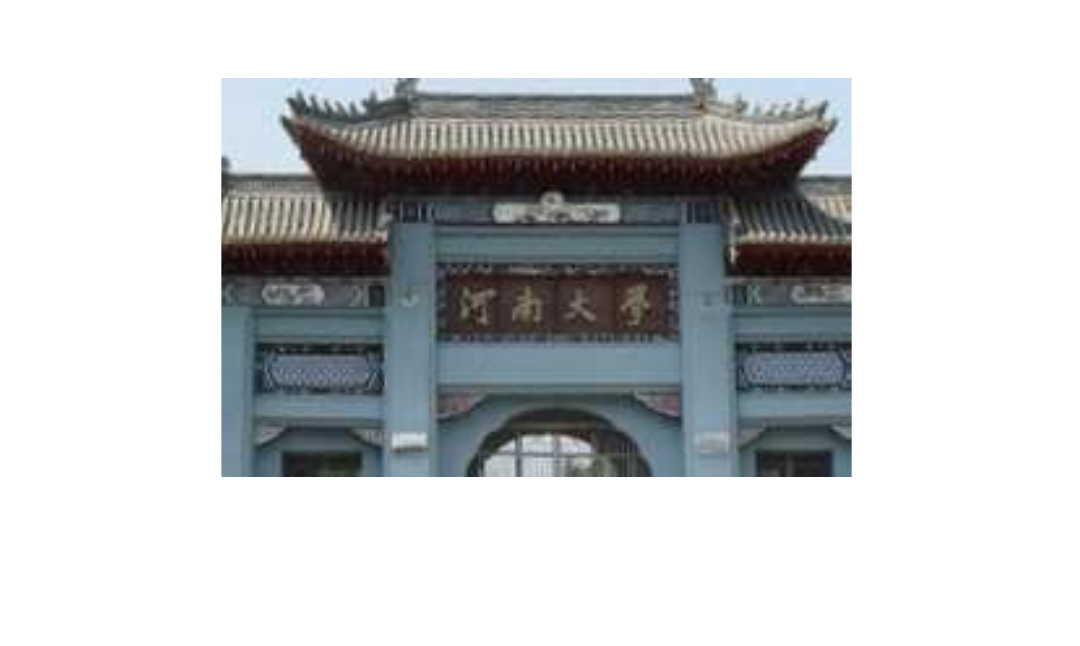}
\caption{{\scriptsize Low-rank textures rectified by algorithms sGS-ADMM and sGS-ADMM\_G. Left: the original images where red windows denote
the original input and green windows denote the deformed texture found by our methods. The remaining columns denote the rectified textures using the transforms found by
sGS-ADMM (middle right) and sGS-ADMM\_G (right), respectively.}}
\label{fig1}
\end{figure}

In the second test, we numerically evaluate the computational improvement of both algorithms based on some representative synthetic and natural low-rank patterns shown in the first row of Figure \ref{fig2}. For these images to be tested, we introduce a small deformation ( say rotation by $10^o$) to each texture, as shown in the third and forth rows of Figure \ref{fig2}. Then we examine whether these algorithms can converge to the correct solution under some random corruptions. For each external loop (Outer), we compare these methods with respect to the number of internal iterations (Iter), the computing time (Time),  the final Rank of the solution $X$ (Rank), the final $\|E\|_1$, and the final KKT residual (Tol). Detailed comparison results are listed in Table \ref{tab1}

\begin{figure}[htbp]
\vspace{-.0cm}\centering
  \includegraphics[scale=0.19] {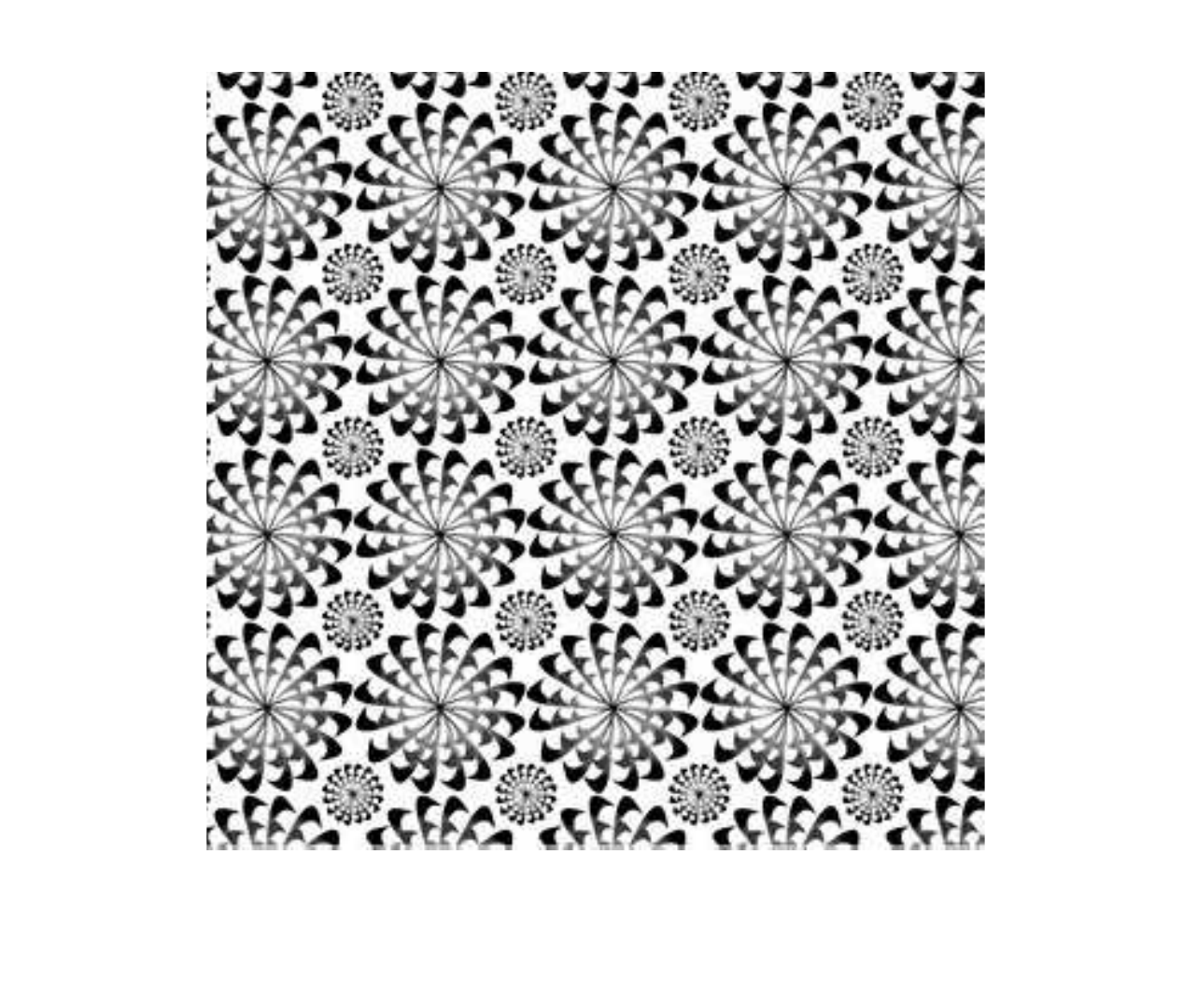}\hspace{-0.4cm}
  \includegraphics[scale=0.12] {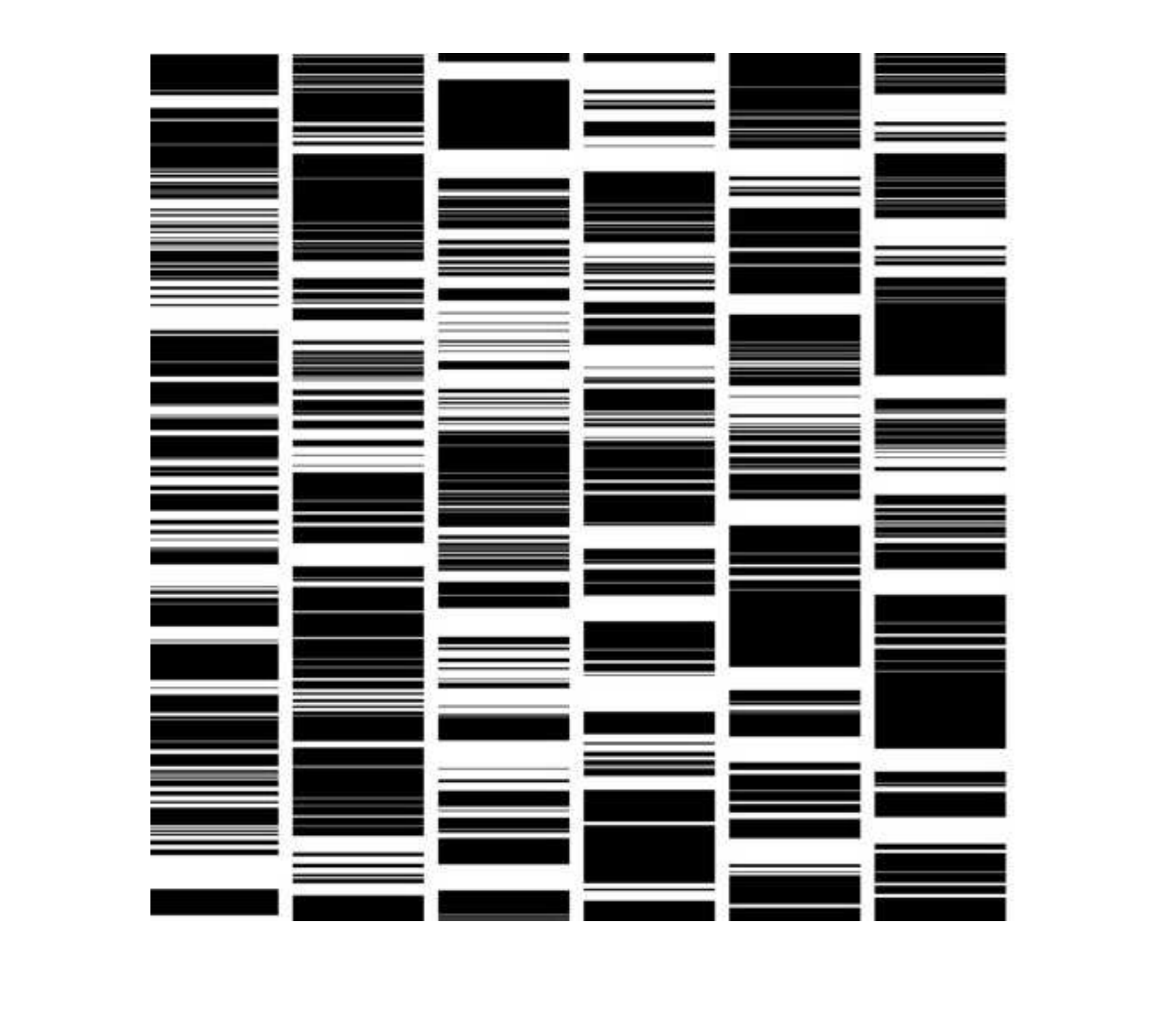}\hspace{-0.4cm}
  \includegraphics[scale=0.14] {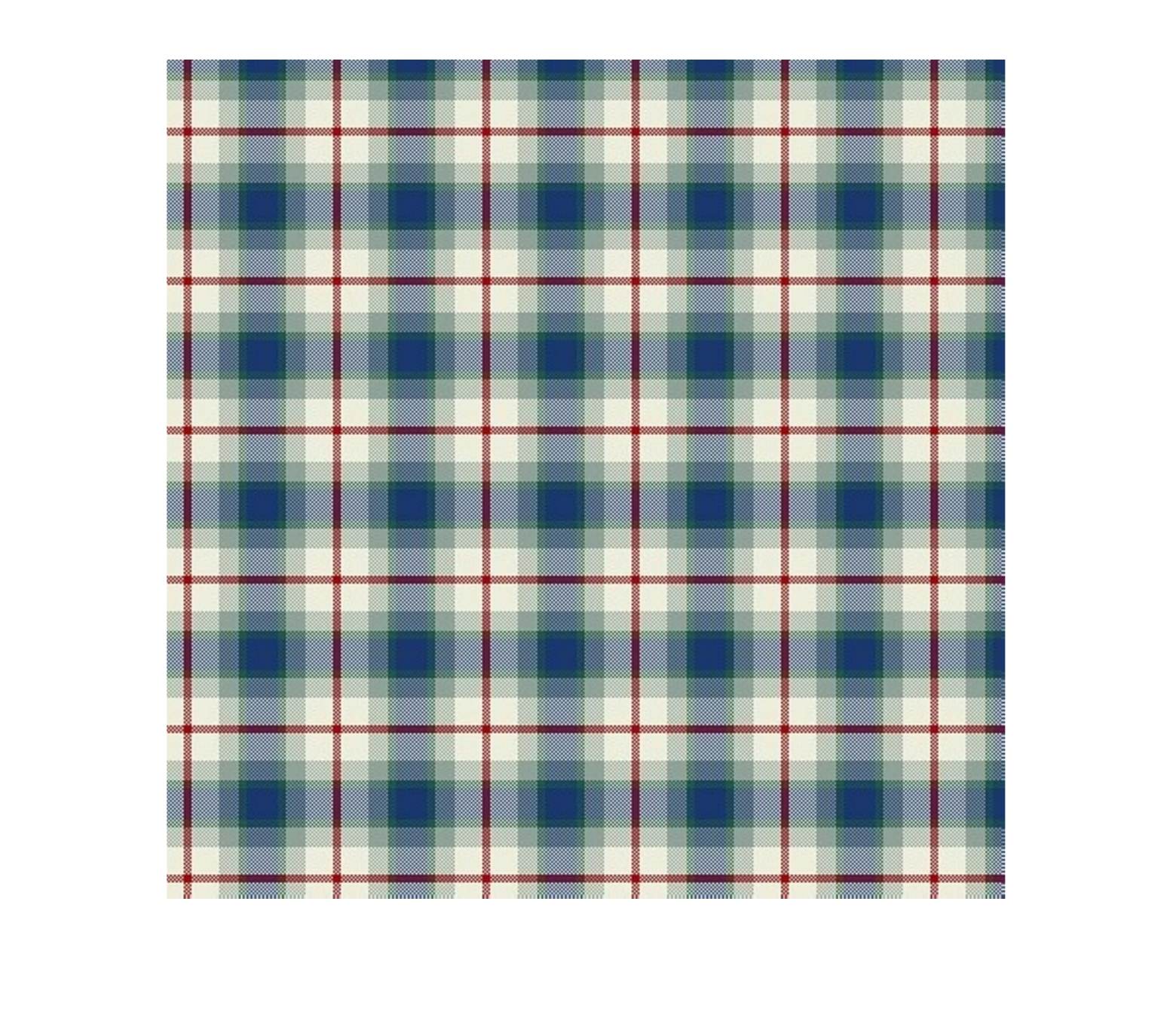}\hspace{-0.4cm}
  \includegraphics[scale=0.145]{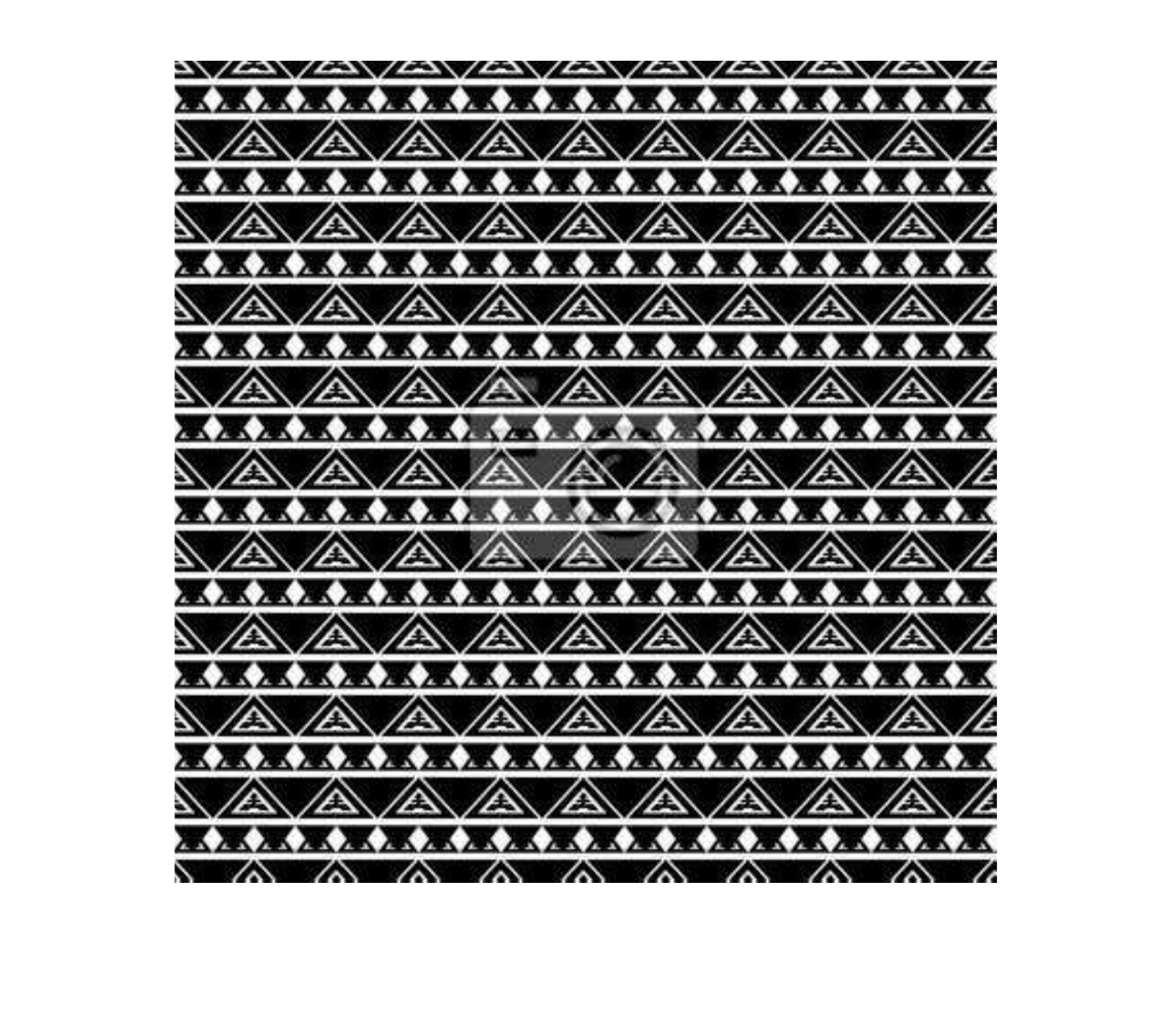}\\
  \vspace{-0.2cm}
  \includegraphics[scale=0.145]{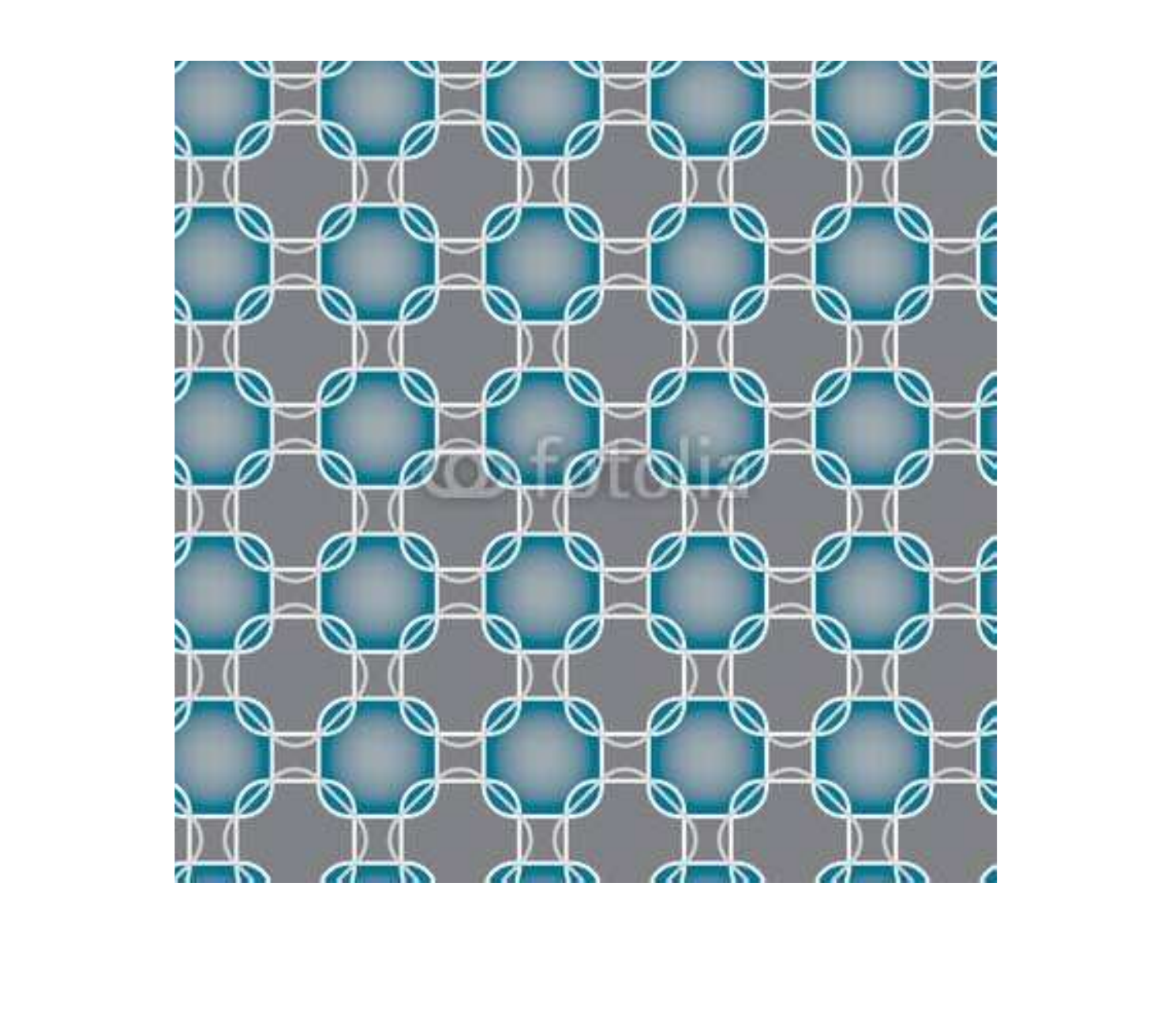}\hspace{-0.4cm}
  \includegraphics[scale=0.14] {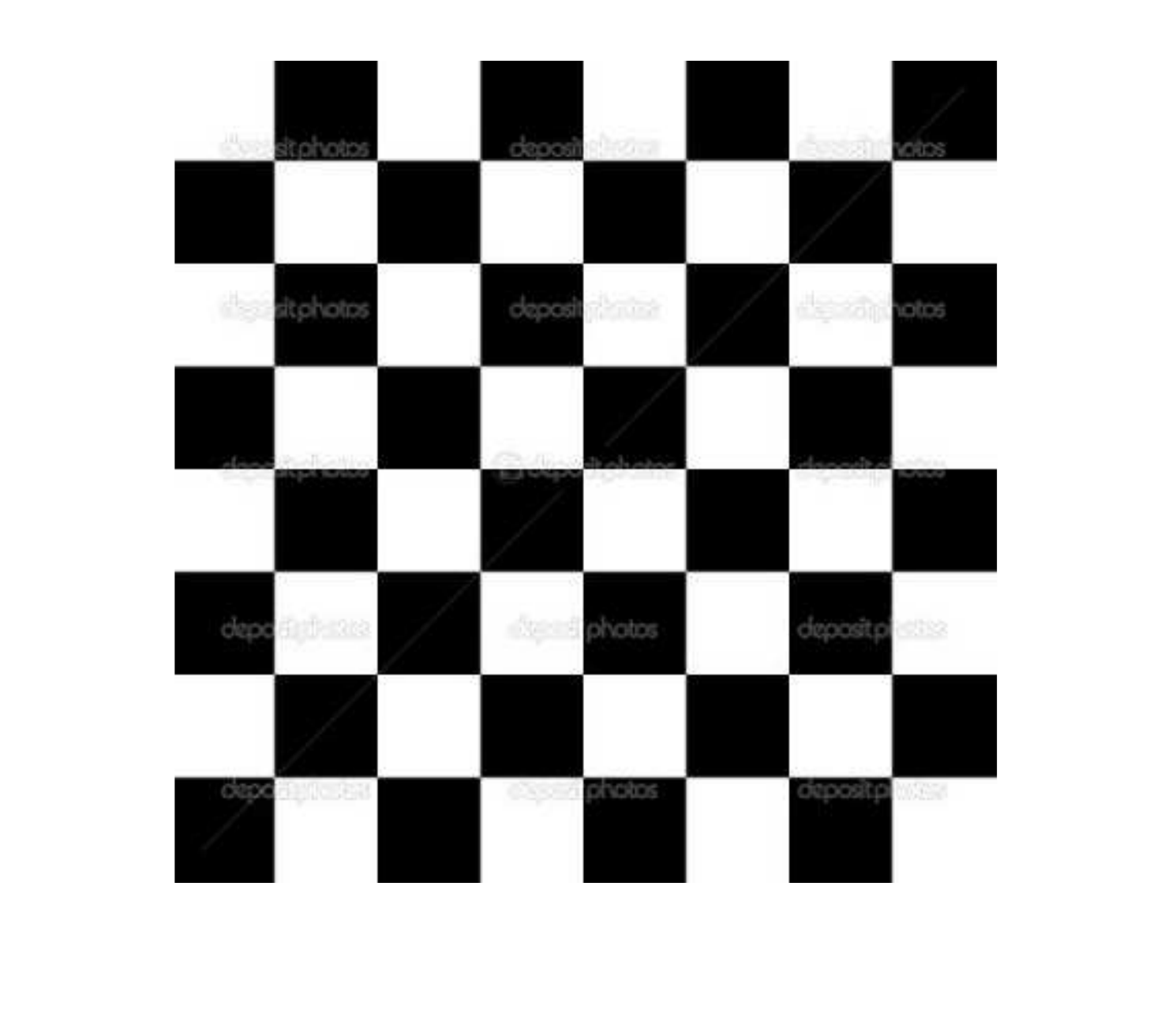}\hspace{-0.4cm}
  \includegraphics[scale=0.185]{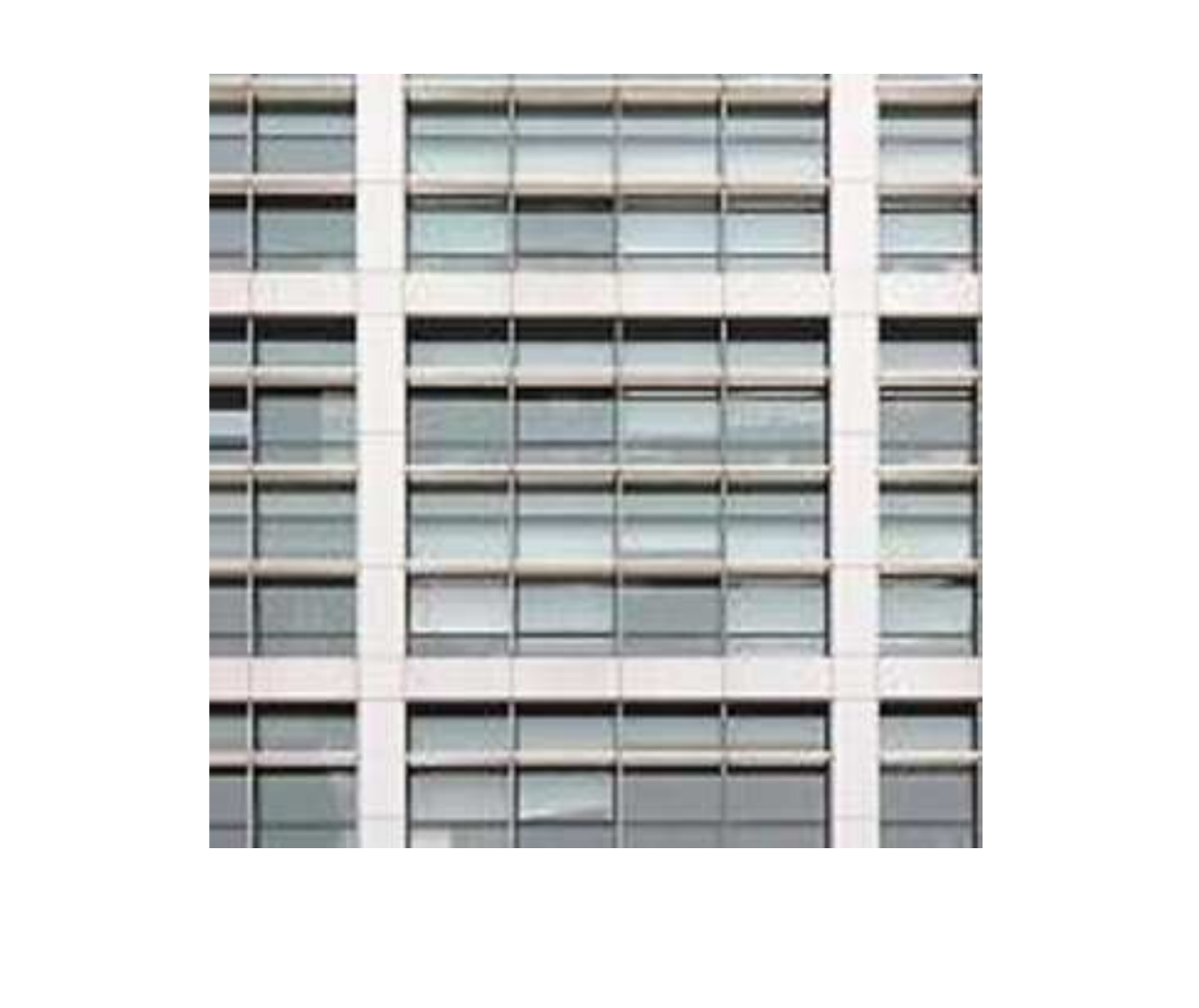}\hspace{-0.4cm}
  \includegraphics[scale=0.14] {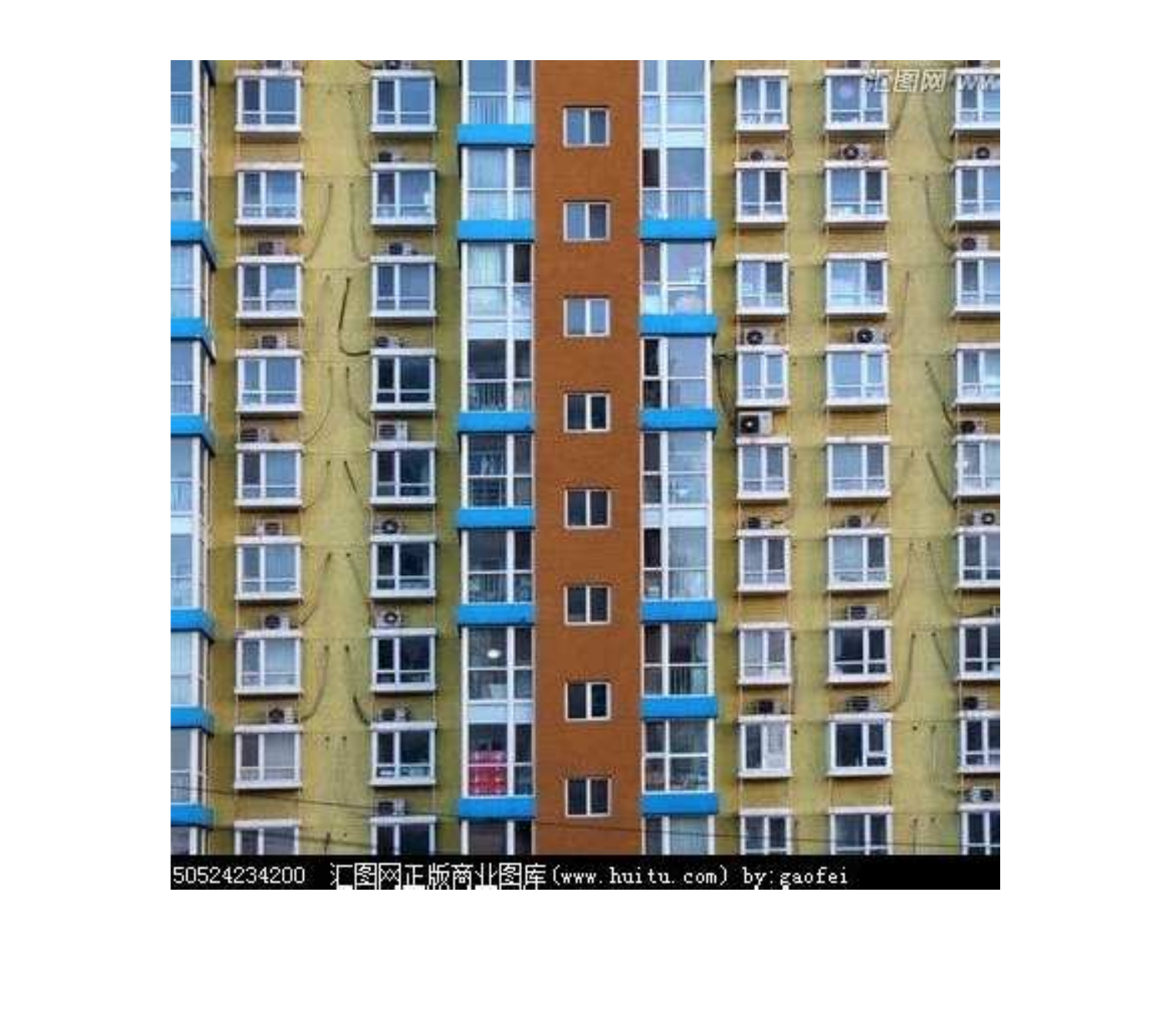}\\\hrule
  \vspace{0.2cm}
  \includegraphics[scale=0.21]{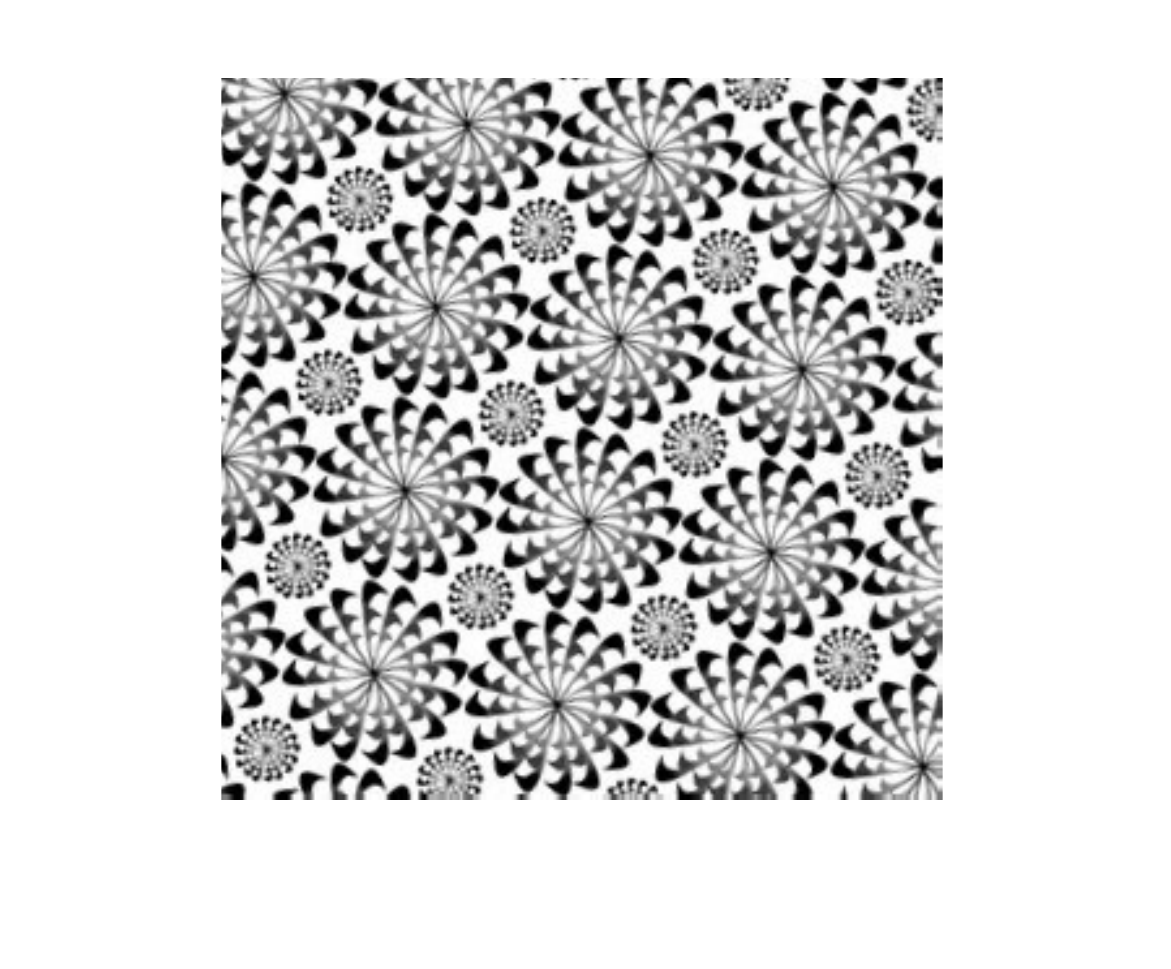}\hspace{-0.4cm}
  \includegraphics[scale=0.14]{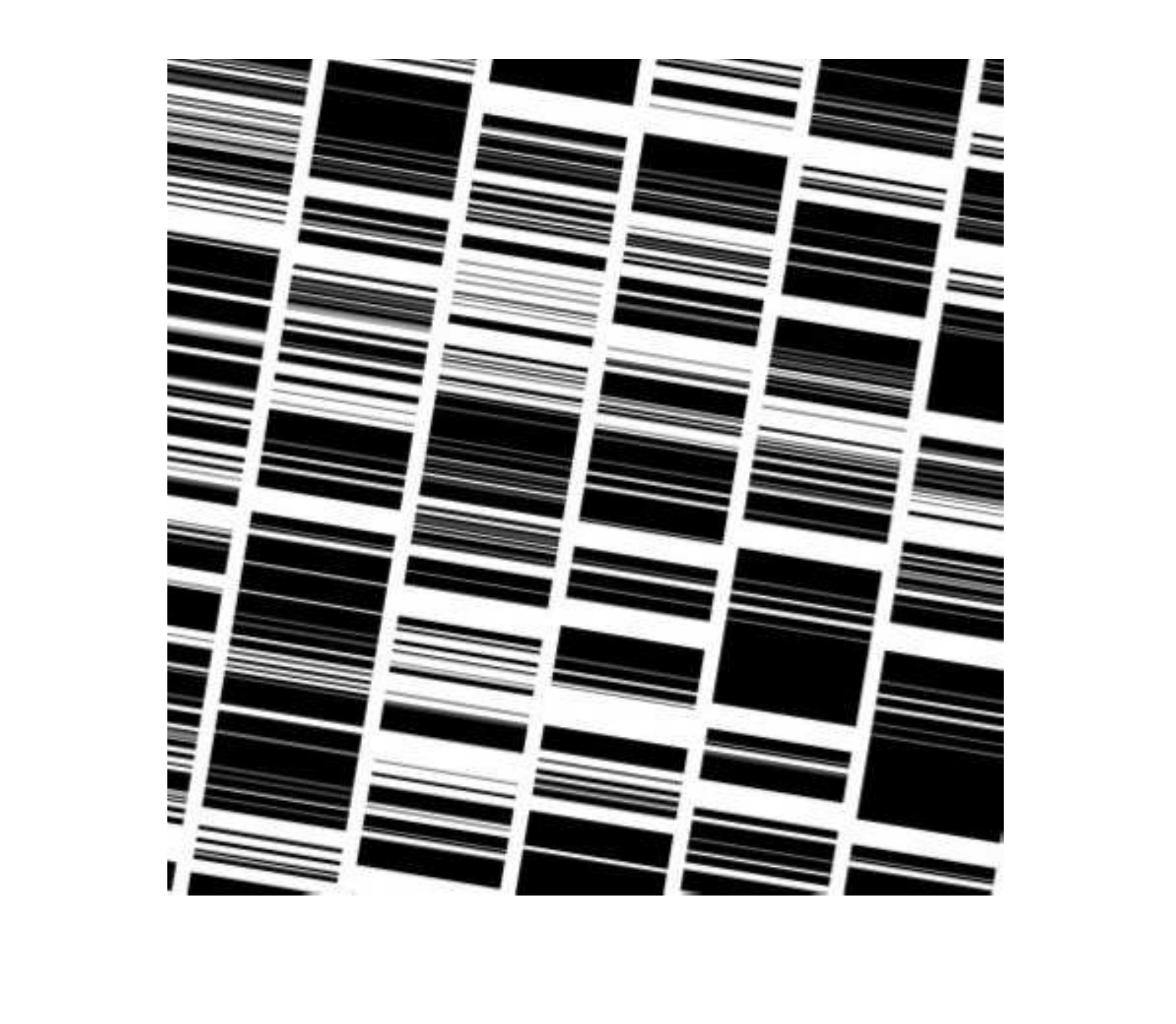}\hspace{-0.4cm}
  \includegraphics[scale=0.15]{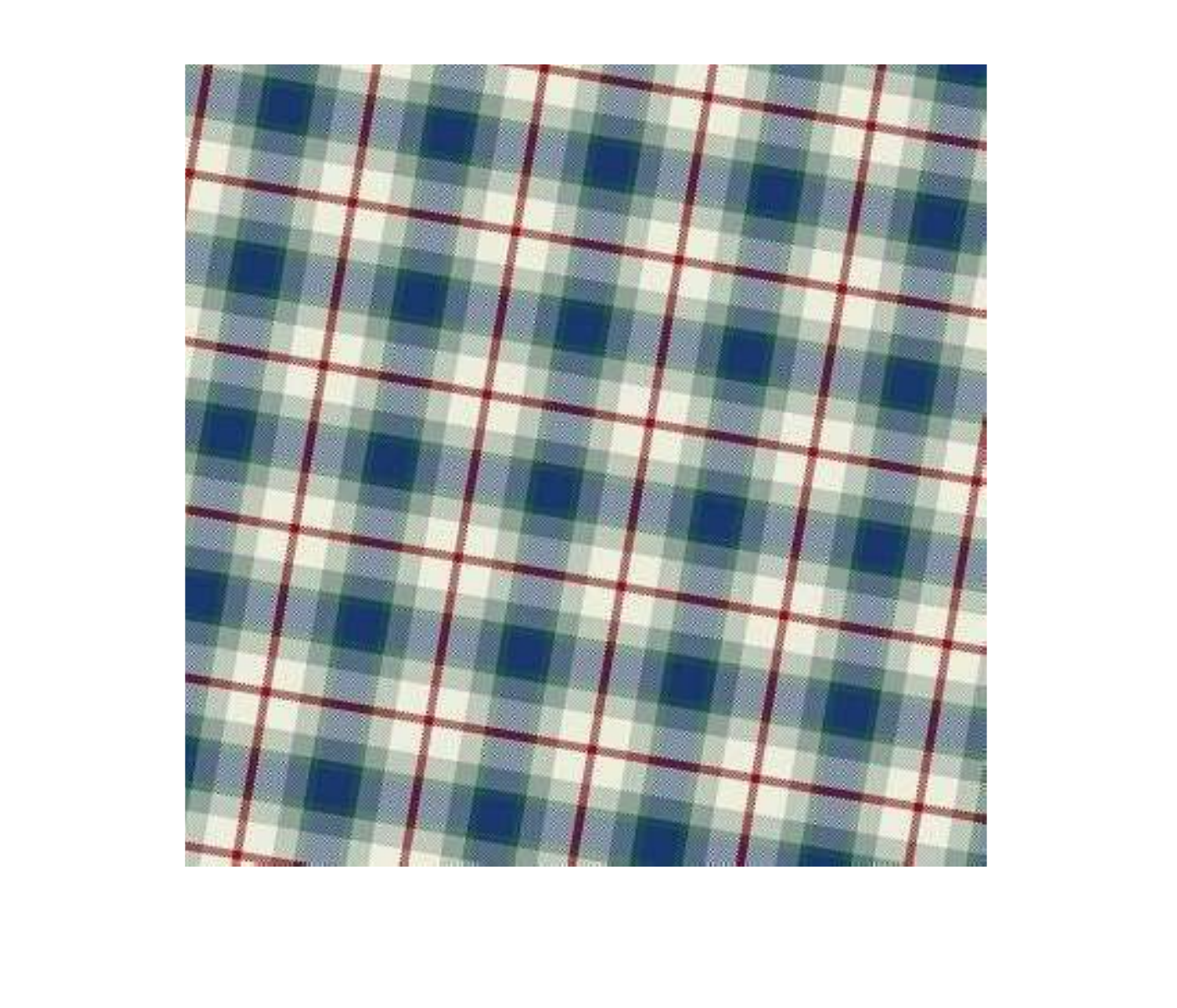}\hspace{-0.4cm}
  \includegraphics[scale=0.16]{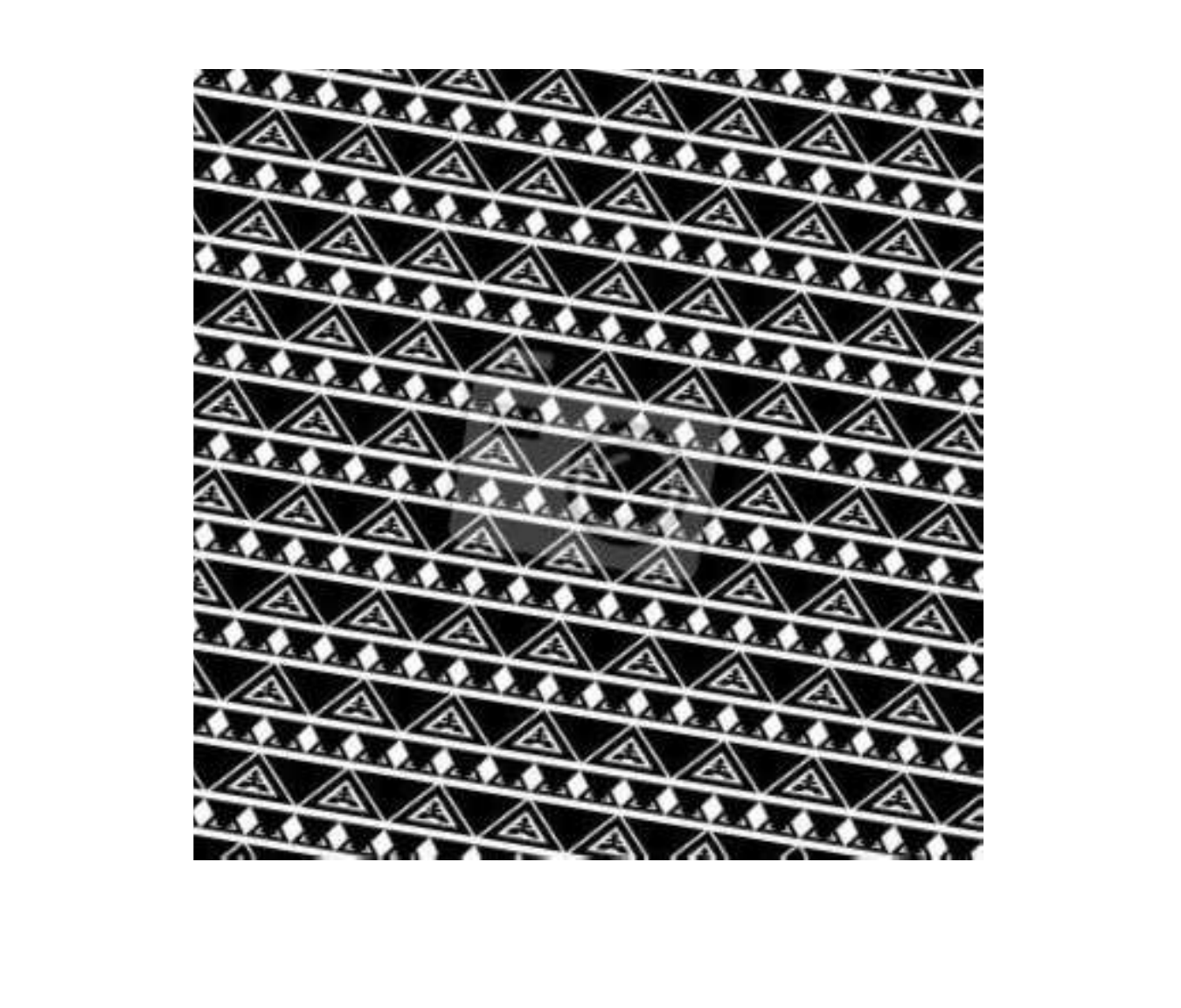}\\
  \vspace{-0.2cm}
  \includegraphics[scale=0.16] {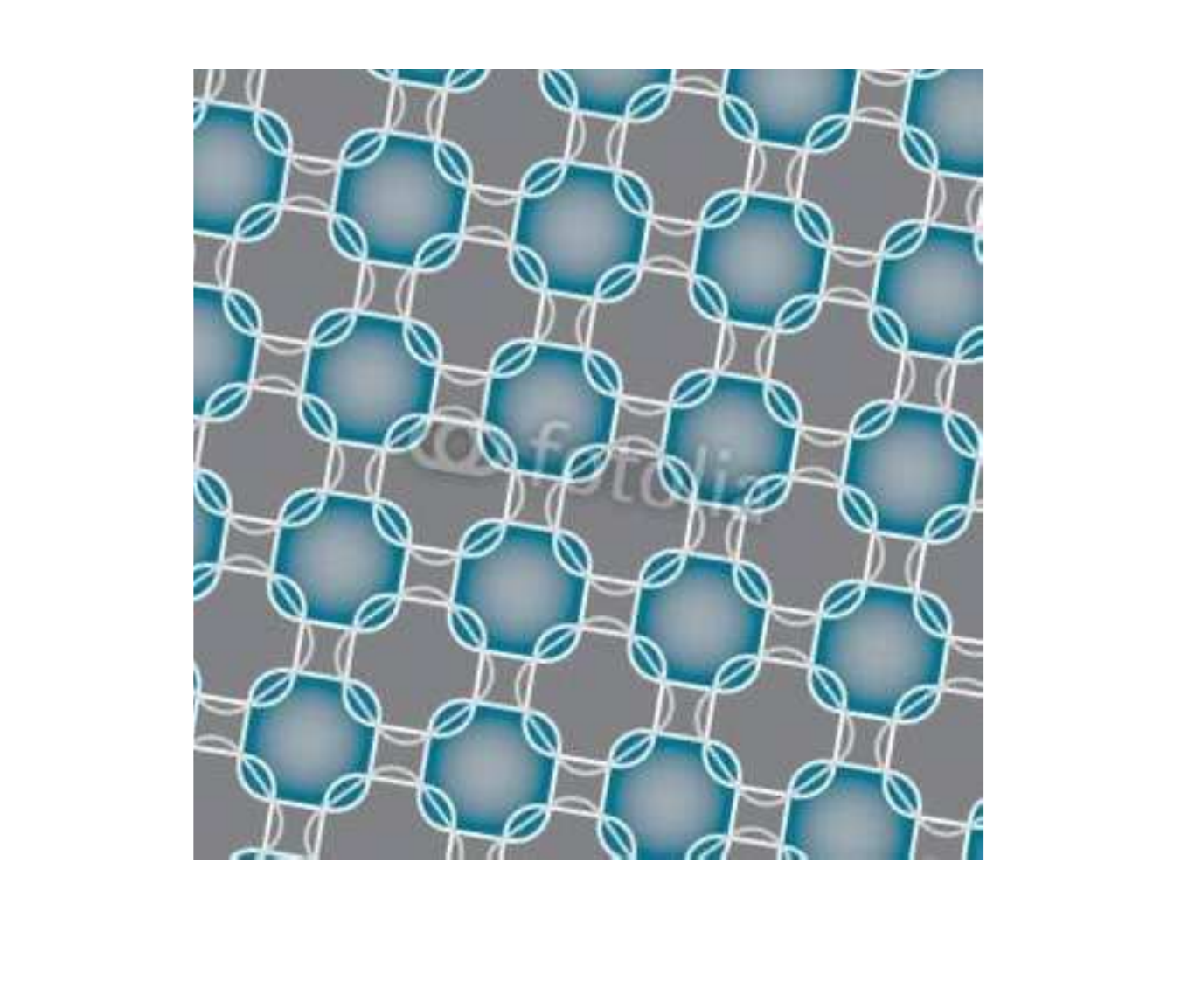}\hspace{-0.4cm}
  \includegraphics[scale=0.16] {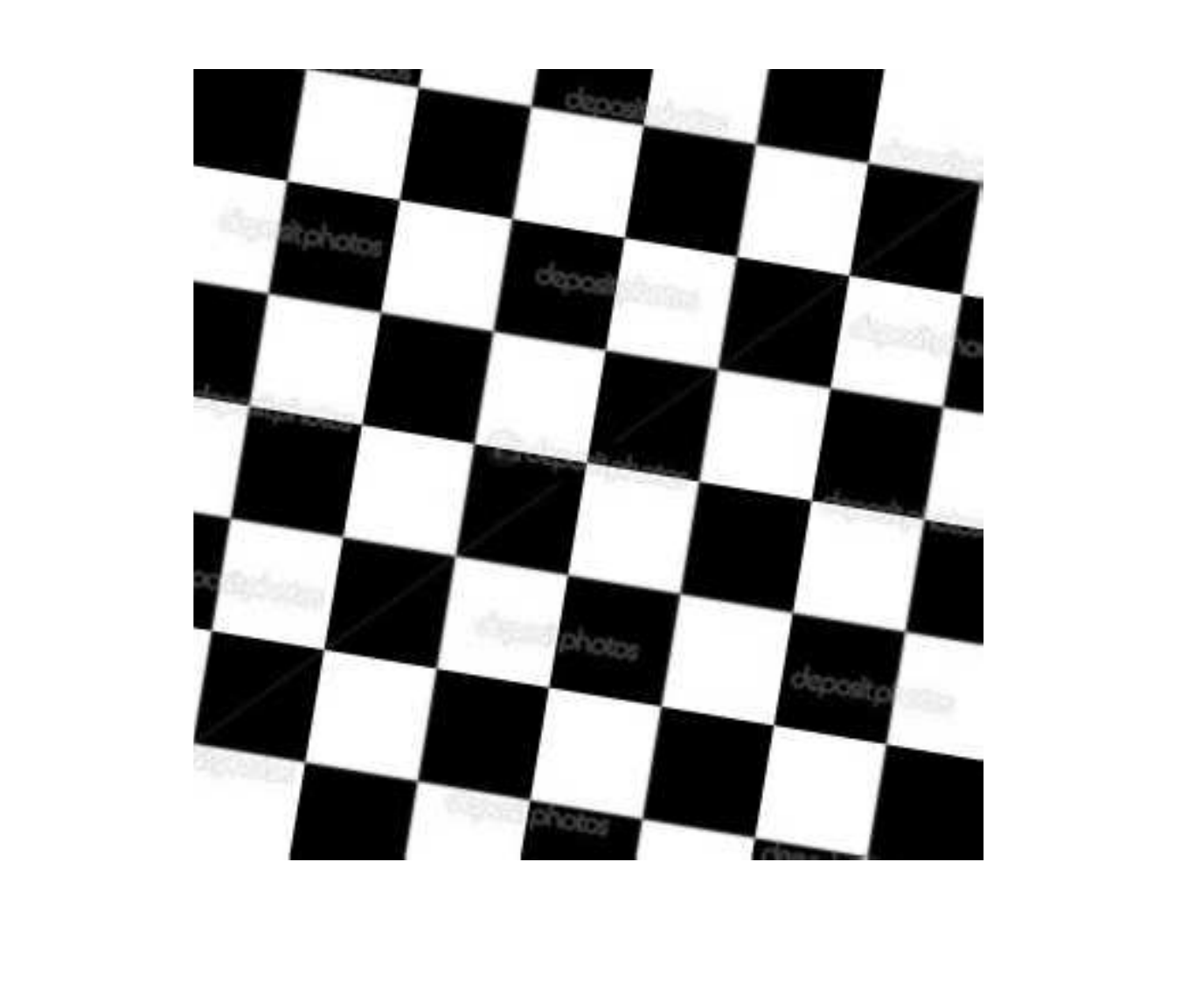}\hspace{-0.4cm}
  \includegraphics[scale=0.195]{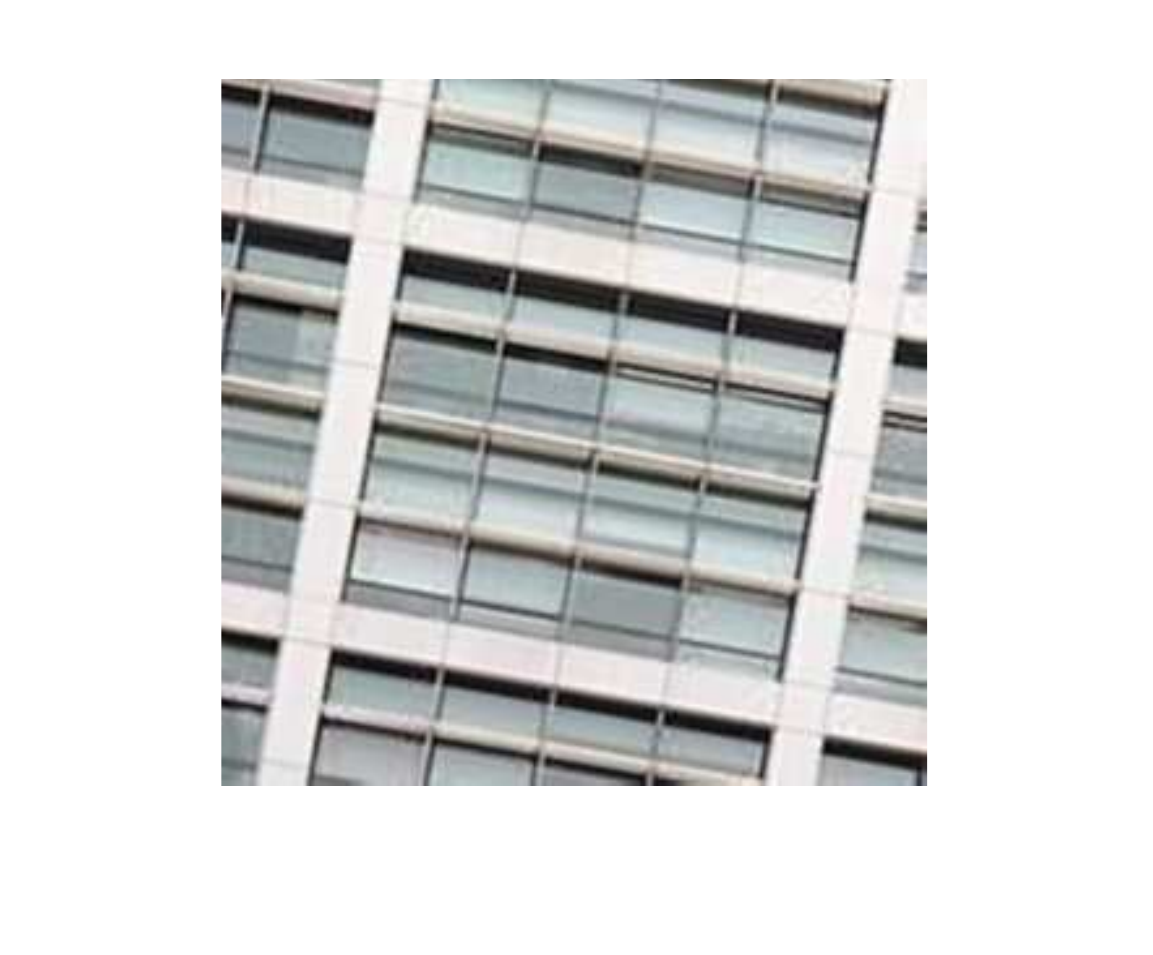}\hspace{-0.4cm}
  \includegraphics[scale=0.16] {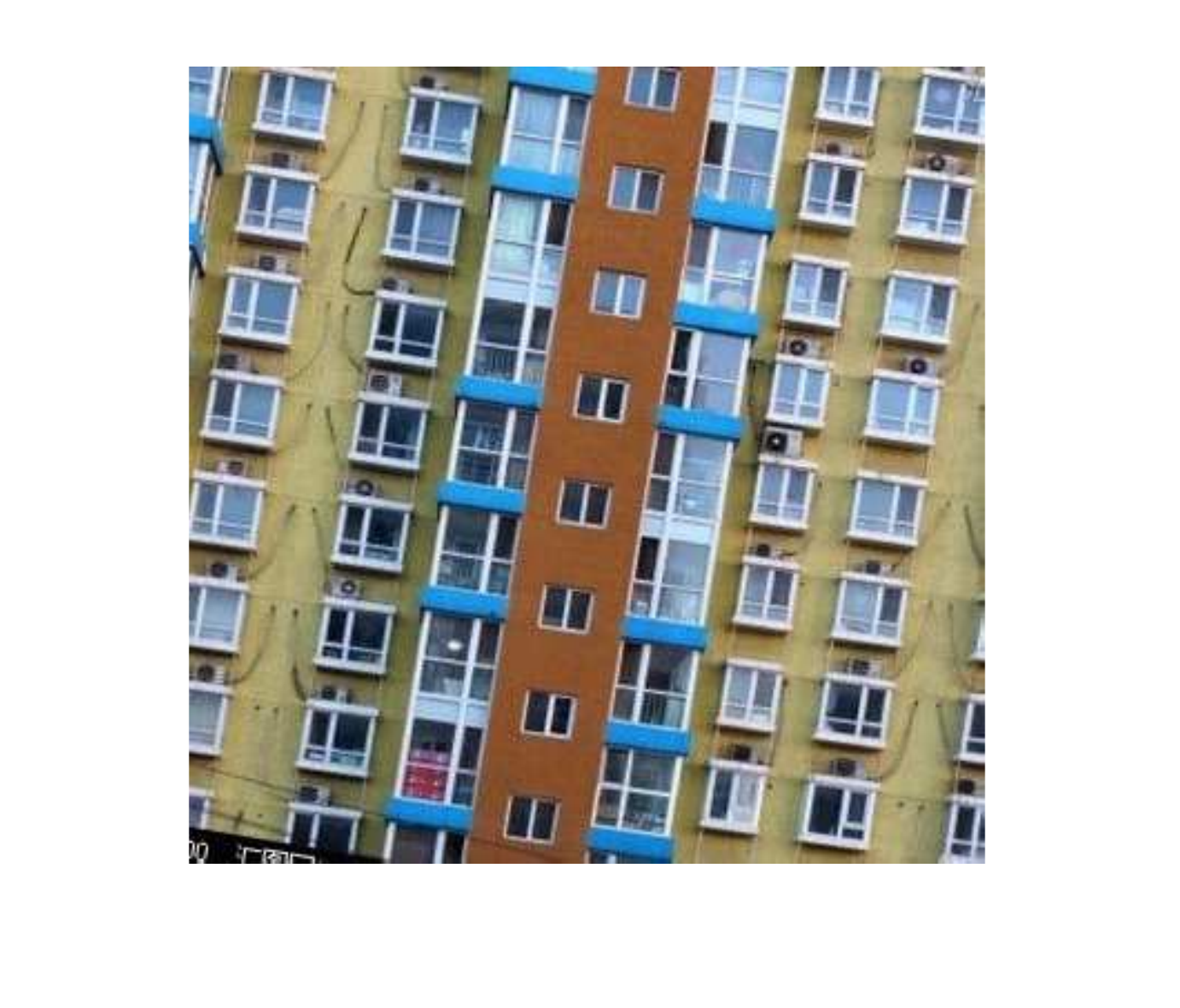}\\ \hrule
  \vspace{0.2cm}
  \includegraphics[scale=0.19]{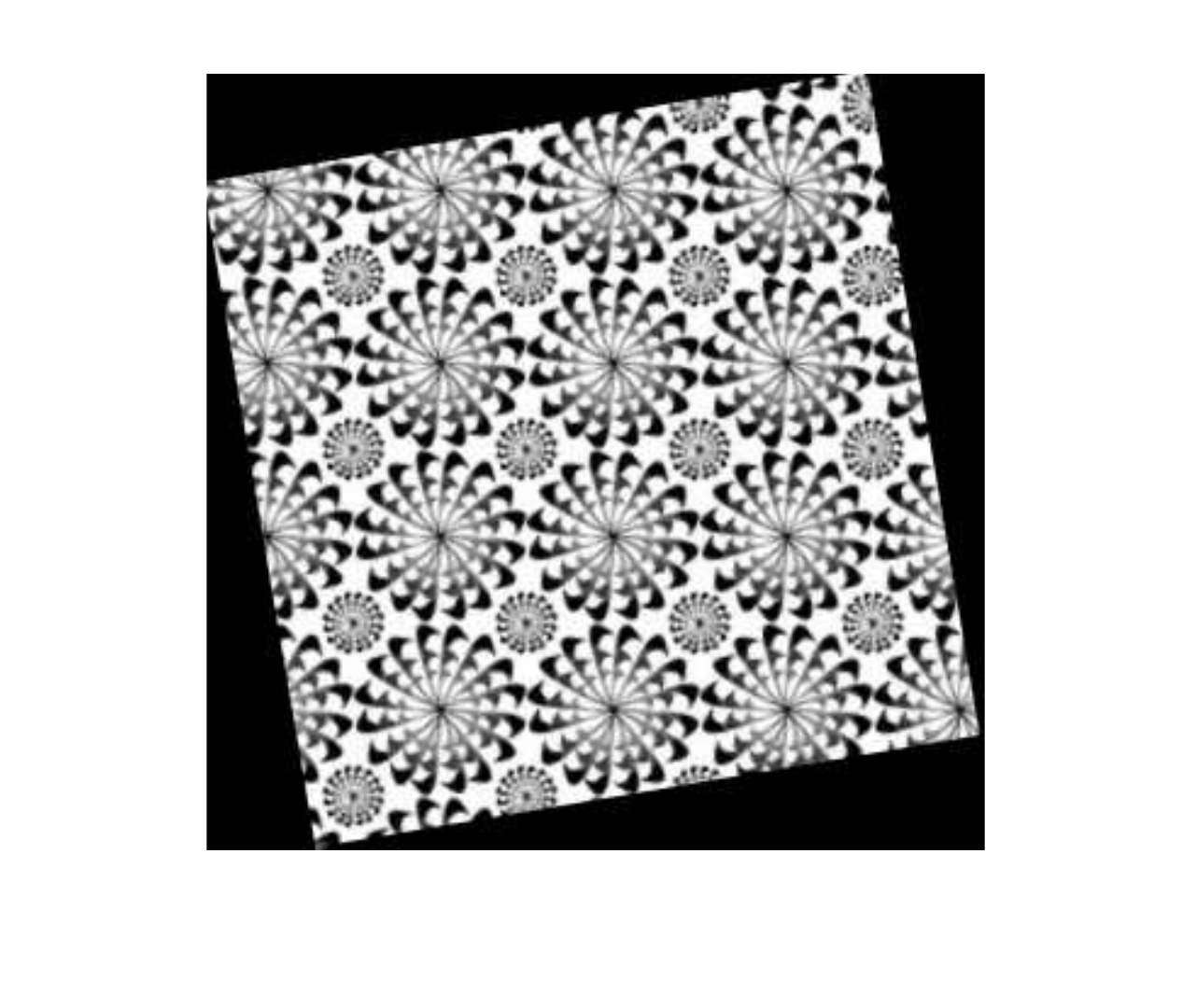}\hspace{-0.4cm}
  \includegraphics[scale=0.12]{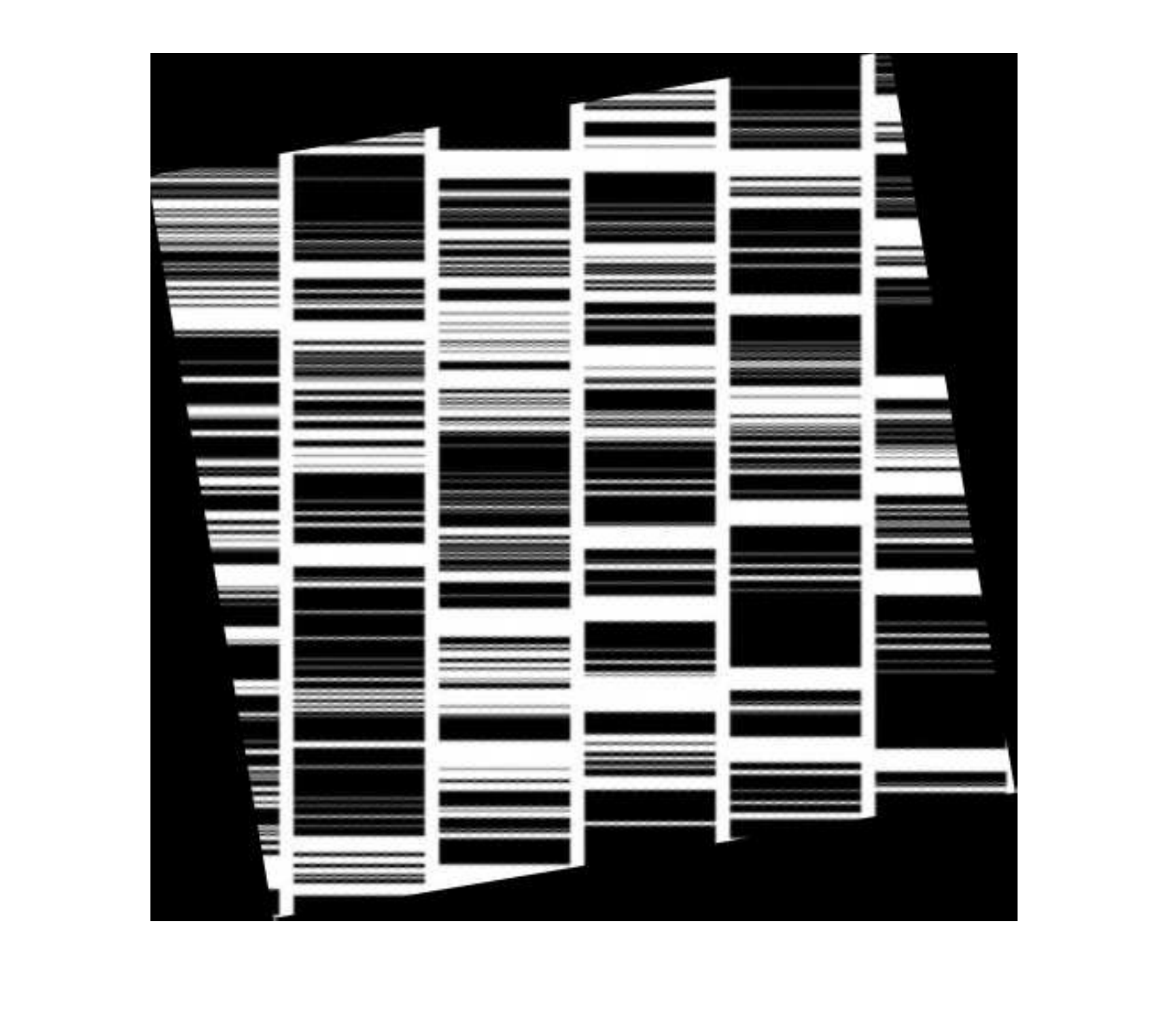}\hspace{-0.4cm}
  \includegraphics[scale=0.13]{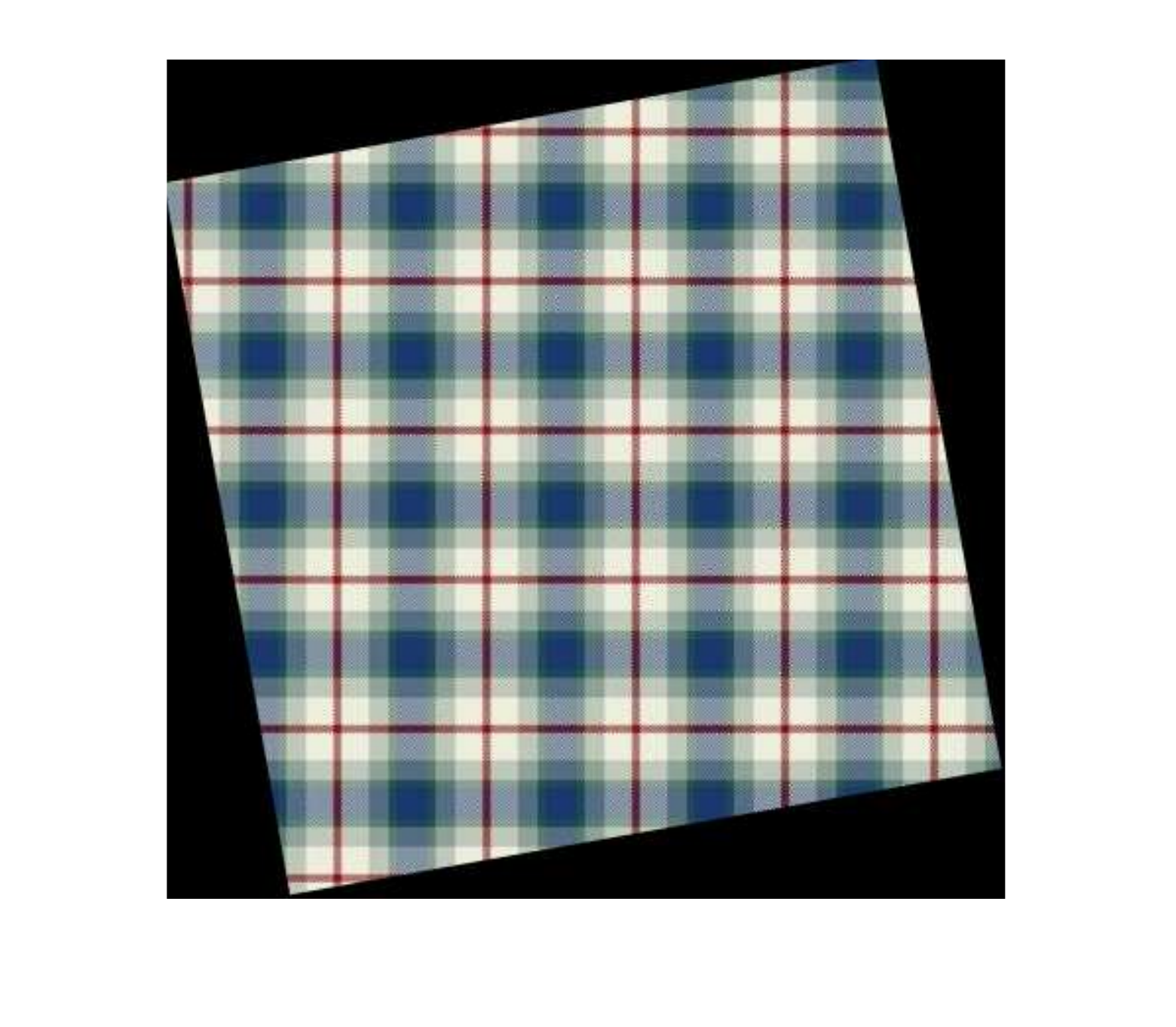}\hspace{-0.4cm}
  \includegraphics[scale=0.14]{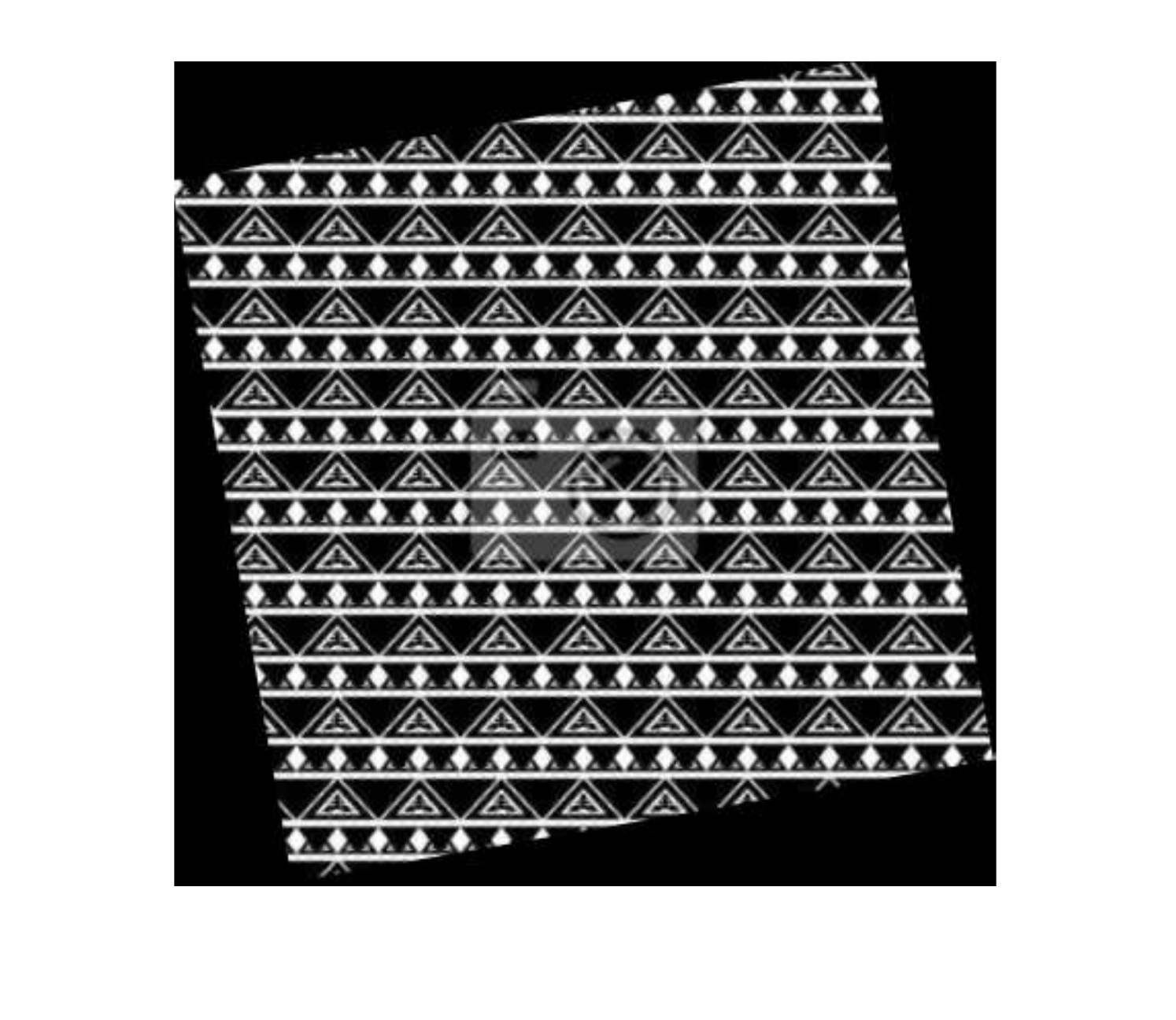}
  \vspace{-0.4cm}
  \includegraphics[scale=0.14]{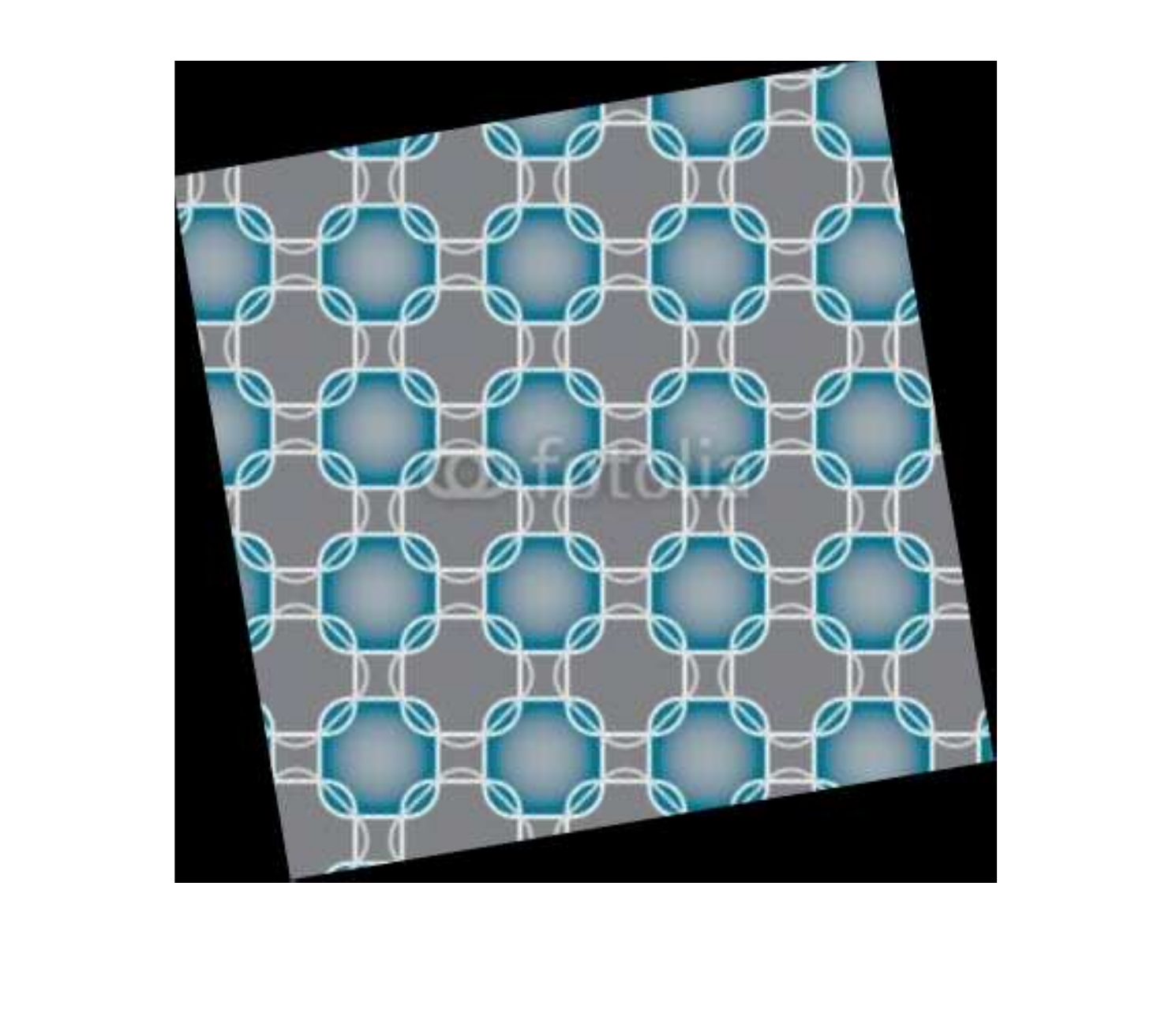}\hspace{-0.4cm}
  \includegraphics[scale=0.14]{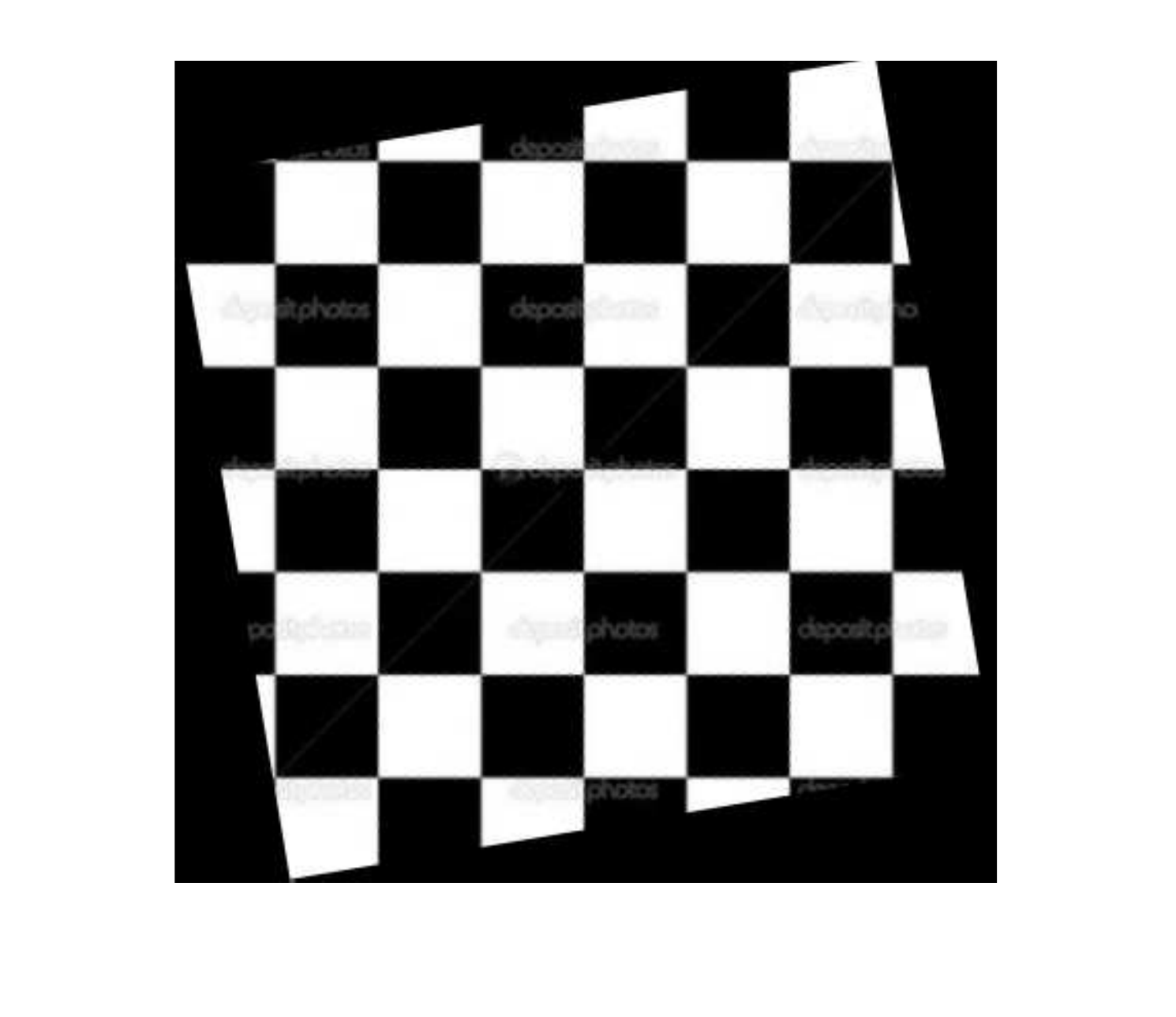}\hspace{-0.4cm}
  \includegraphics[scale=0.18]{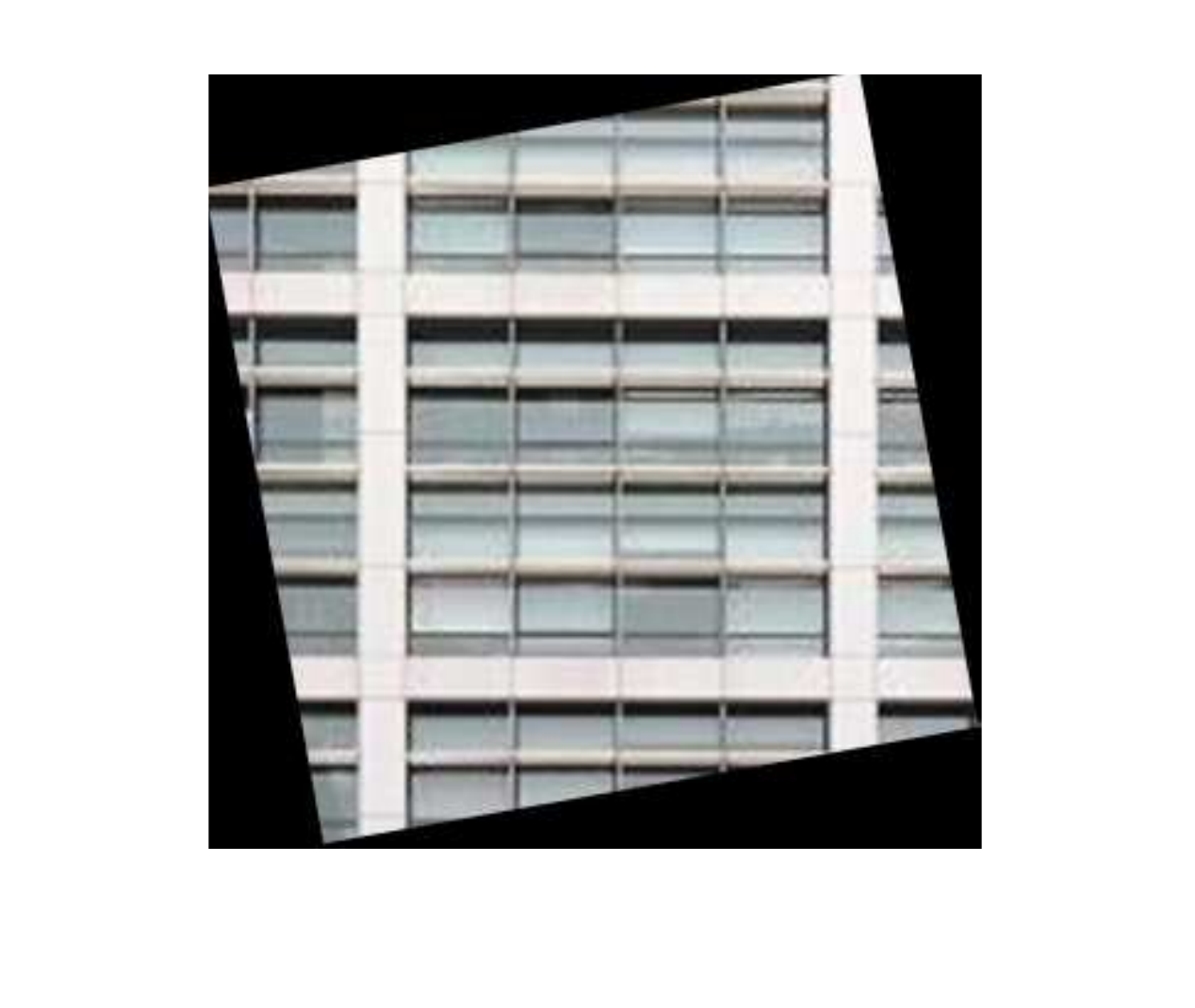}\hspace{-0.4cm}
  \includegraphics[scale=0.14]{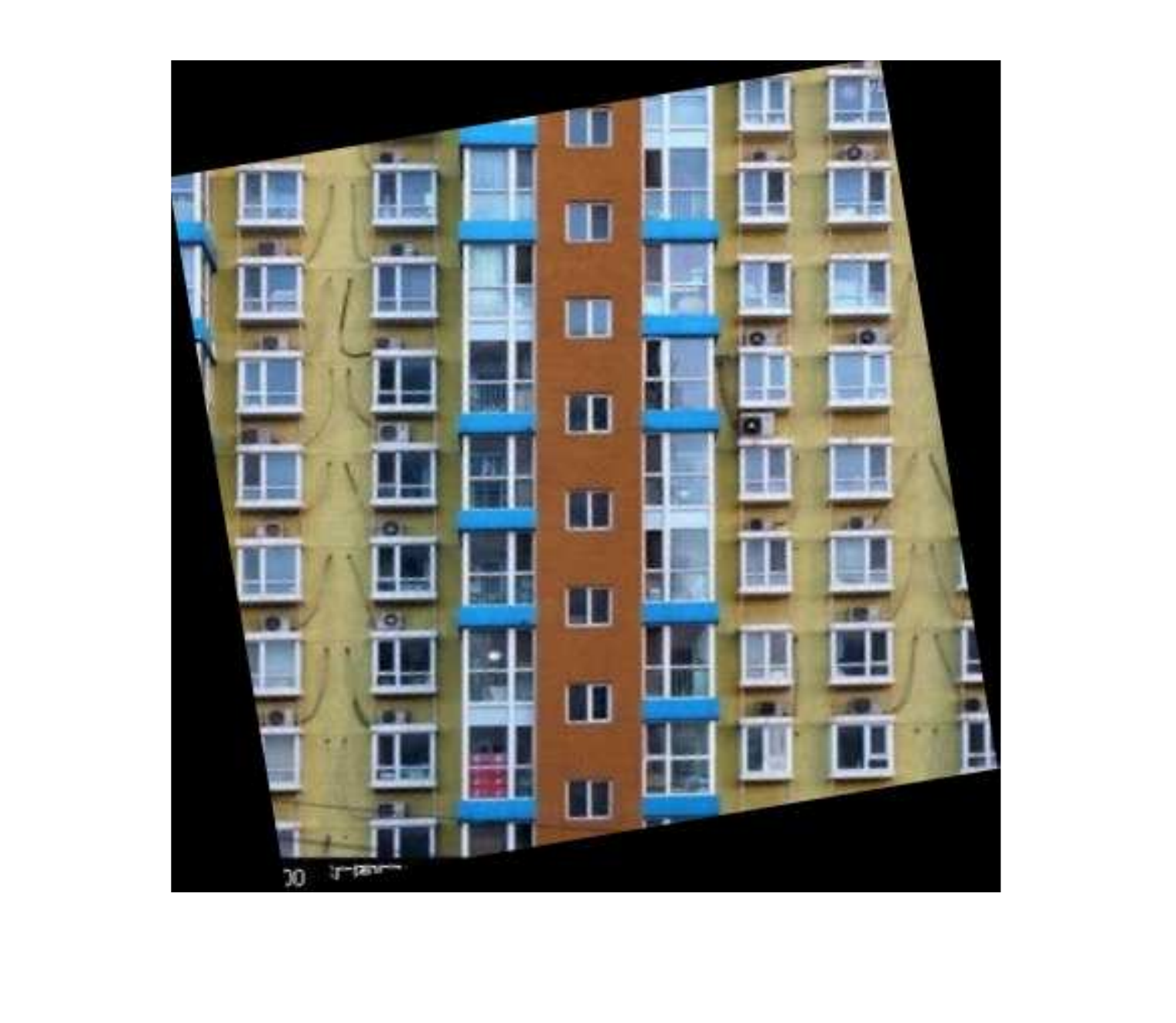}
\caption{{\scriptsize Representative Results of sGS-ADMM. First two rows: original regular low-rank patterns and textures; Middle two rows: rotated images; Bottom two rows: rectified image by sGS-ADMM.}}
\label{fig2}
\end{figure}

\begin{table*}[ht]
{\scriptsize \centering \caption{Comparison results of TILT with sGS-ADMM and sGS-ADMM\_G.}
\begin{tabular}{|c|c||c|c|c|c|c|c|c|c|c|c|c|c|c|c|c|}
\hline \multicolumn{2}{|c||}{}  & \multicolumn{5}{c|}{TILT} &\multicolumn{5}{c|}{sGS-ADMM} & \multicolumn{5}{c|}{sGS-ADMM\_G}  \\
\hline
No.& Outer & Iter  & Time & Rank & $\|E\|_1$  & Tol & Iter  & Time & Rank & $\|E\|_1$  & Tol  & Iter  & Time & Rank& $\|E\|_1$  & Tol  \\
\hline
1    &1  &  628&  1.84&    14& 1.06e+00& 9.99e-04&   389&  1.03&    14& 1.06e+00& 9.96e-04&  350&  0.95&    14& 1.06e+00& 9.96e-04\\
     &2  &  773&  2.34&    13& 5.56e-01& 9.99e-04&   479&  1.20&    13& 5.56e-01& 9.97e-04&  430&  1.08&    13& 5.56e-01& 9.99e-04\\
     &3  &  799&  2.06&    13& 5.38e-01& 1.00e-03&   495&  1.30&    13& 5.39e-01& 9.98e-04&  445&  1.16&    13& 5.38e-01& 9.99e-04\\
     &4  &  813&  2.22&    13& 5.16e-01& 9.99e-04&   503&  1.34&    13& 5.16e-01& 9.99e-04&  452&  1.14&    13& 5.16e-01& 1.00e-03\\
     &5  &  801&  2.16&    13& 5.20e-01& 9.96e-04&   496&  1.33&    13& 5.20e-01& 9.93e-04&  446&  1.16&    13& 5.20e-01& 9.96e-04\\
\hline
2    &1  &  800&  2.05&    13& 7.42e-01& 1.00e-03&   495&  1.30&    13& 7.42e-01& 9.99e-04& 445&  1.17&    13& 7.42e-01& 9.99e-04\\
     &2  &  930&  2.45&    12& 4.16e-01& 1.00e-03&   576&  1.45&    12& 4.16e-01& 9.99e-04& 518&  1.27&    12& 4.15e-01& 9.99e-04\\
     &3  &  912&  2.45&    11& 4.04e-01& 9.98e-04&   564&  1.39&    11& 4.04e-01& 9.97e-04& 506&  1.25&    11& 4.04e-01& 9.99e-04\\
     &4  &  931&  2.53&    11& 4.04e-01& 1.00e-03&   576&  1.48&    11& 4.04e-01& 1.00e-03& 519&  1.31&    11& 4.04e-01& 9.99e-04\\
     &5  &  923&  2.31&    11& 4.04e-01& 9.99e-04&   571&  1.45&    11& 4.04e-01& 9.99e-04& 514&  1.30&    11& 4.03e-01& 9.99e-04\\
\hline
3    &1  &   987&  2.45&    13& 1.87e-01& 9.99e-04&   610&  1.58&    13& 1.87e-01& 1.00e-03& 549&  1.45&    13& 1.87e-01& 9.98e-04\\
     &2  &  1365&  3.25&     8& 7.39e-02& 9.99e-04&   844&  2.16&     8& 7.39e-02& 9.99e-04& 760&  1.97&     8& 7.38e-02& 9.99e-04\\
     &3  &  1393&  3.42&     8& 6.09e-02& 9.98e-04&   861&  2.08&     8& 6.08e-02& 1.00e-03& 774&  1.97&     8& 6.09e-02& 9.96e-04\\
     &4  &  1381&  3.41&     8& 6.11e-02& 9.97e-04&   853&  2.09&     8& 6.11e-02& 9.99e-04& 764&  1.94&     8& 6.12e-02& 9.99e-04\\
\hline
4    &1  &  573&  1.50&    15& 9.09e-01& 9.94e-04&    355&  0.92&    15& 9.09e-01& 9.93e-04&  319&  0.88&    15& 9.09e-01& 9.95e-04\\
     &2  &  712&  1.75&    13& 6.20e-01& 9.99e-04&    440&  1.14&    13& 6.20e-01& 1.00e-03&  396&  1.00&    13& 6.20e-01& 9.96e-04\\
     &3  &  694&  1.75&    13& 5.89e-01& 9.99e-04&    429&  1.08&    13& 5.89e-01& 9.99e-04&  386&  1.09&    13& 5.89e-01& 9.98e-04\\
     &4  &  732&  1.81&    13& 5.99e-01& 9.99e-04&    453&  1.22&    13& 5.99e-01& 9.98e-04&  407&  1.09&    13& 5.99e-01& 1.00e-03\\
     &5  &  707&  1.72&    13& 5.96e-01& 9.99e-04&    437&  1.13&    13& 5.96e-01& 1.00e-03&  394&  1.13&    13& 5.96e-01& 9.90e-04\\
     &6  &  708&  1.78&    13& 5.97e-01& 1.00e-03&    438&  1.14&    13& 5.97e-01& 9.99e-04&  394&  1.03&    13& 5.97e-01& 9.99e-04\\
\hline
5    &1  &  857&  2.17&    13& 4.35e-01& 9.99e-04&   530&  1.36&    13& 4.35e-01& 9.99e-04&   477&  1.19&    13& 4.35e-01& 9.99e-04 \\
     &2  &  958&  2.39&    12& 3.75e-01& 9.91e-04&   592&  1.47&    12& 3.75e-01& 9.94e-04&   533&  1.28&    12& 3.75e-01& 9.91e-04 \\
     &3  &  984&  2.41&    11& 3.67e-01& 9.99e-04&   609&  1.50&    11& 3.67e-01& 9.99e-04&   548&  1.53&    11& 3.67e-01& 9.99e-04 \\
     &4  &  972&  2.38&    12& 3.65e-01& 9.99e-04&   602&  1.50&    12& 3.64e-01& 9.99e-04&   542&  1.50&    12& 3.62e-01& 1.00e-03 \\
     &5  &     &      &      &         &         &   603&  1.50&    12& 3.67e-01& 9.91e-04&   542&  1.47&    12& 3.66e-01& 9.92e-04 \\
\hline
6    &1  &  714&  2.30&    12& 3.98e-01& 9.97e-04 &   442&  1.14&    12& 3.99e-01& 9.94e-04& 397&  1.16&    12& 3.99e-01& 9.98e-04\\
     &2  &  816&  2.14&    10& 1.91e-01& 9.98e-04 &   505&  1.22&    10& 1.91e-01& 9.99e-04& 455&  1.28&    10& 1.91e-01& 9.96e-04\\
     &3  &  847&  2.11&     9& 1.91e-01& 1.00e-03 &   524&  1.25&     9& 1.91e-01& 1.00e-03& 472&  1.27&     9& 1.91e-01& 1.00e-03\\
     &4  &  841&  2.05&     9& 1.91e-01& 9.99e-04 &   521&  1.28&     9& 1.91e-01& 9.99e-04& 468&  1.25&     9& 1.91e-01& 1.00e-03\\
\hline
7    &1  &  918&  2.31&    12& 4.33e-01& 9.99e-04&   568&  1.53&    12& 4.33e-01& 9.99e-04 & 511&  1.31&    12& 4.33e-01& 1.00e-03 \\
     &2  &  954&  2.34&    13& 3.31e-01& 9.96e-04&   590&  1.50&    13& 3.31e-01& 1.00e-03 & 530&  1.34&    13& 3.31e-01& 9.99e-04 \\
     &3  &  990&  2.45&    12& 3.13e-01& 9.99e-04&   612&  1.53&    12& 3.13e-01& 9.99e-04 & 551&  1.41&    12& 3.13e-01& 9.98e-04 \\
     &4  &  992&  2.45&    12& 3.15e-01& 9.99e-04&   614&  1.56&    12& 3.15e-01& 9.99e-04 & 553&  1.38&    12& 3.14e-01& 9.99e-04 \\
\hline
8    &1  &  630&  1.80&    15& 7.03e-01& 9.97e-04&   390&  1.05&    15& 7.03e-01& 9.98e-04  &  351&  0.94&    15& 7.03e-01& 9.98e-04 \\
     &2  &  602&  1.66&    16& 6.34e-01& 9.97e-04&   373&  0.97&    16& 6.34e-01& 9.96e-04  &  336&  0.84&    16& 6.34e-01& 9.99e-04 \\
     &3  &  614&  1.59&    16& 6.27e-01& 9.97e-04&   380&  1.00&    16& 6.27e-01& 9.97e-04  &  342&  0.88&    16& 6.28e-01& 9.96e-04 \\
     &4  &  621&  1.56&    16& 6.25e-01& 9.96e-04&   384&  1.00&    16& 6.25e-01& 9.97e-04  &  346&  0.88&    16& 6.25e-01& 9.96e-04 \\ \hline
\end{tabular}\label{tab1}
}
\end{table*}

AS can be seen from Table \ref{tab1} that, all the algorithms obtained the same final rank and the comparable KKT residuals for all test cases. From these results, we also see that sGS-ADMM and sGS-ADMM\_G are the most competitive while TILT is the slowest one. We also observe that the total number of internal iterations of both sGS-ADMM and sGS-ADMM\_G is much smaller than TILT. This is because sGS-ADMM and sGS-ADMM\_G sweep the variable $\Delta\tau$ twice during each iteration which in turn ensures the convergence of the iterative process.  We also test a series of other images contained with different decompositions and we observed the consistent results. These results and observations clearly demonstrated the efficiency and stability of sGS-ADMM and sGS-ADMM\_G.

\section{Conclusions}\label{last}
Transform Invariant Low-Rank Textures targets to extract both textural and geometric information defining regions of low-rank planar patterns from 2D scene. This task can
be characterized as a sequence of matrix nuclear-norm and $\ell_1$-norm involvednon-smooth convex minimization problems. The extended directly ADMM implemented by Zhang {\itshape et al.} \cite{TILT} often performs numerically well, but its theoretical convergence is not guaranteed. In this paper, we employed the sGS-ADMM developed by Li, Sun \& Toh \cite{XDLIMP} for solving the convex non-smooth model (\ref{modelthree}) and used the sGS decomposition technique in \cite{SGSTH} to extend the generalized ADMM to solve (\ref{modelthree}).  The distinct feature of the sGS-ADMM is that it needs to  update the variable $\Delta\tau$ again, but it greatly reduces the total number of iterations and computing time. The reason for the improved performance is the using the novel sGS decomposition theorem in \cite{SGSTH} to establish the equivalence between the internal iterative scheme and the classical ADMM with an addition of a particular proximal term. With the encouraging numerical performance of sGS-ADMM, it is worthwhile to investigate other techniques, such as the rank-correction technique \cite{MPS}, to further improve the solution accuracy. This is an interesting topic of future research.

\section*{Acknowledgements}

We would like to thank the anonymous referees and the associate editor for their useful comments and suggestions
which improved this paper greatly. We are very grateful to Dr. Liang Chen at Hunan University for helpful discussions on the optimality conditions and stopping criterion
used in the algorithms' implementations.

\end{document}